\documentclass[12pt]{article}
\usepackage{a4wide}
\usepackage[english]{babel}
\usepackage{appendix}
\usepackage{amsmath}
\usepackage{amsfonts}
\usepackage{amsthm}
\usepackage{amssymb}
\usepackage{psfrag}
\newcommand{\dx}{\,\text{\rm d}\Omega}
\newcommand{\ds}{\,\text{\rm d}\Gamma}
\newcommand{\mcT}{\mathcal{T}}

\usepackage{todonotes}
\usepackage{algorithm,algorithmicx,algpseudocode}
\usepackage{subcaption}
\usepackage{stmaryrd}

\newtheorem{theorem}{Theorem}[section]
\newtheorem{lemma}[theorem]{Lemma}

\theoremstyle{definition}

\newtheorem{remark}{Remark}[section]

\newcommand{\Vhzero}{V_h(\Omega_0)}
\newcommand{\VhOmega}{V_h(\Omega)}

\numberwithin{equation}{section}
\newcommand{\bndmap}{\mathcal{M}}
\newcommand{\bfI}{\boldsymbol{I}}
\newcommand{\mcL}{\mathcal{L}}
\newcommand{\mcF}{\mathcal{F}}
\newcommand{\mcTh}{{\mathcal{T}_h}}
\newcommand{\mcFh}{{\mathcal{F}_h}}
\newcommand{\mcO}{\mathcal{O}}
\newcommand{\IR}{\mathbb{R}}
\newcommand{\DtV}{D_{t,\vartheta}}
\newcommand{\DOV}{D_{\Omega,\vartheta}}
\newcommand{\POV}{\partial_{\Omega,\vartheta}}
\newcommand{\divv}{\nabla \cdot}
\newcommand{\divs}{\nabla_\Gamma \cdot}
\newcommand{\Honed}{{[H^1_0(\Omega_0)]^d}}

\newcommand{\tn}{|||}
\newcommand{\Gammafix}{\Gamma_{\rm{fix}}}
\DeclareMathOperator*{\argmin}{arg\,min}

\title{\bf A Cut Finite Element Method for the Bernoulli Free Boundary Value Problem
\thanks{This research was supported in part by the Swedish Foundation for Strategic Research Grant No.\ AM13-0029, the Swedish Research Council Grants Nos.\ 2011-4992, 2013-4708, the Swedish Research Programme Essence, and EPSRC, UK, Grant Nr. EP/J002313/1. }
}
\author{Erik Burman\footnote{Department of Mathematics, University 
 College London, Gower Street, London WC1E 6BT, UK, e.burman@ucl.ac.uk}
 \mbox{ } 
Daniel Elfverson\footnote{Department of Mathematics and Mathematical Statistics, Ume{\aa} University, SE--901~87~~Ume{\aa}, Sweden,daniel.elfverson@umu.se}
 \mbox{ }
 Peter Hansbo\footnote{Department of Mechanical Engineering, J\"onk\"oping University, SE-551~11 J\"onk\"oping, Sweden, peter.hansbo@ju.se } 
\\
 \mbox{ } 
{Mats G. Larson}\footnote{Department of Mathematics and Mathematical Statistics, Ume{\aa} University, SE--901~87~~Ume{\aa}, Sweden, mats.larson@umu.se}
\mbox{ } 
{Karl Larsson}\footnote{Department of Mathematics and Mathematical Statistics, Ume{\aa} University, SE--901~87~~Ume{\aa}, Sweden, karl.larsson@umu.se}
}

\begin{document}
\maketitle
\begin{abstract}
We develop a cut finite element method for the Bernoulli free boundary 
problem. The free boundary, represented by an approximate signed distance function on a fixed background mesh, is allowed to intersect elements in an arbitrary fashion. This leads to so called cut elements in the vicinity of 
the boundary. To obtain a stable method, stabilization terms is added in the vicinity  of the cut elements penalizing the gradient jumps across element sides. The stabilization also ensures good conditioning of the resulting discrete system. We develop a method for shape optimization based on 
moving the distance function along a velocity field which is computed as the $H^1$ Riesz representation of the shape derivative. We show that the velocity field is the solution to an interface problem and we prove an a priori error estimate of optimal order, given the limited regularity 
of the velocity field across the interface, for the the velocity field in the $H^1$ norm. Finally, we present illustrating numerical results.
\end{abstract}
\section{Introduction} 
\label{sec:introduction}
In this paper we consider the application of the recently developed cut finite
element method (CutFEM) \cite{BCHLM15,BH12} to the Bernoulli free boundary
problem. This problem appears in a variety of applications such as stationary
water waves (Stokes waves) and the optimal insulation problem. The Bernoulli
free boundary problem is very well understood from the mathematical point of
view, see \cite{Beu57,FR97,ST08} and the references therein, and it also serves as a
standard test problem for different optimization and numerical methods, see
e.g. \cite{Har08,Hou95,KGPT07} among others. When numerically solving free
boundary problems it is highly beneficial to avoid updating the computational
mesh when updating the boundary, since large motions of the boundary may
require complete remeshing. This can be achieved by the use of fictitious
domain methods, which, however, are known not to perform very well for shape
optimization or free surface problems due to their lack of accuracy close to
the boundary \cite{Ha08}. An exception is the least squares formulation 
suggested in \cite{EHM08}, where similarly as in \cite{BH12} the formulation is restricted to the physical domain. 

A fictitious domain method which does not lose accuracy close to the 
boundary is the recently developed CutFEM method, see \cite{BCHLM15}. CutFEM 
uses weak enforcement of the boundary conditions and a sufficiently accurate representation of the domain together with certain 
consistent stabilization terms to guarantee stability, optimal accuracy, and conditioning independent of the position of the boundary in the background 
mesh. Furthermore, no expensive and complicated mesh operations (edge split, 
edge collapse, {remeshing}, etc.) need to be performed when updating the 
boundary. This is a significant gain especially for complicated boundaries 
and for 3D applications. CutFEM has successfully been applied for problems 
with unknown or moving boundaries in, e.g., \cite{CHL15,HLZ15}.
 
As in \cite{ADF14,AJT04} we consider a shape optimization approach to solve
the Bernoulli free boundary problem using sensitivity analysis and a level set
representation \cite{OS88} to track the evolution of the free boundary. In the
sensitivity analysis we do not use the standard \emph{Hadamard structure} of
the shape functional, i.e., we do not express the shape functional as a normal
perturbation of the boundary. Instead we use a volume representation which
requires less smoothness and has proved to possess certain superconvergence properties compared to the boundary formulation \cite{HPS15}, see also \cite{HP15,LS15}. To obtain a velocity field from the
shape derivate, we use the Hilbertian regularization suggested in \cite{Go06}, 
where we essentially let the velocity field on the domain be defined as the 
solution to the weak elliptic problem associated with the $H^1$ inner product 
with right hand side given by the shape derivative functional. We may thus 
view the velocity field as the $H^1$ Riesz representation of the shape 
derivative functional acting on $H^1$. This procedure leads to an elliptic 
interface problem for the velocity field. The free boundary is then updated by 
moving the level set along the velocity field.

We derive a priori error estimates for the CutFEM approximation of the 
primal problem and the dual problem, involved in the computation of the shape derivative in the $W^1_p$, $2\leq p < \infty$, norm, and then we use these 
estimates to prove an a priori error estimate for the discrete approximation of the velocity field in the $H^1$ norm. In the error estimates for the primal and 
dual problems we use inverse estimates, for $2<p<\infty$, which leads to 
suboptimal convergence rates, but it turns out these bounds are indeed 
sharp enough to prove optimal order estimates for the discrete velocity 
field.

An outline of the paper is as follows: in Section~\ref{model_problem} we
present the model problem and CutFEM discretization, in Section~\ref{shape_derivative} we use sensitivity analysis to derive the shape
derivative, in Section~\ref{velocity_field} we discuss how to compute a
regularized descent direction from the shape derivative, in
Section~\ref{level_set} we present a level set representation of the free
boundary and method for computing its evolution, in
Section~\ref{optimization_Algorithm} we present an optimization algorithm, 
in Section~\ref{a_priori} we present a priori error estimates of the CutFEM
approximation of the primal and dual problems as well as for the discrete approximation of the velocity field, and finally, in
Section~\ref{numerical_examples} we present numerical experiments to verify
the convergence rates and overall behavior of the optimization algorithm.


\section{Model Problem and Finite Element Method} 
\label{model_problem}
\subsection{Model Problem}
We consider the Bernoulli free boundary value problem:
\begin{alignat}{3}
    -\Delta u &= f&\qquad & \text{in }\Omega \label{eq:Bernoulli-a}\\
    u &= g_D & \qquad & \text{on }\partial\Omega \label{eq:Bernoulli-b}\\
    n \cdot\nabla u &= g_N &\qquad & \text{on }\Gamma \label{eq:Bernoulli-c}
\end{alignat}
where $\Gamma \subset \partial \Omega$ is the free and $\partial \Omega
\setminus \Gamma$ is the fixed part of the boundary, $g_D=0$ and $g_N$ is constant
on $\Gamma$. Note the double boundary conditions on $\Gamma$. We seek to
determine the domain $\Omega$ such that there exist a solution $u$ to
\eqref{eq:Bernoulli-a}-\eqref{eq:Bernoulli-c}. For $f=0$ we refer to \cite{Beu57,FR97,ST08} 
and the references therein for theoretical background of the Bernoulli free
boundary value problem and for $f=0$ we assume that $f$ is such that there
exist a unique solution.

In order to obtain a formulation which is suitable as a starting point 
for a numerical algorithm we recast the overdetermined boundary 
value problem as a constrained minimization problem as follows. We
seek to minimize the functional 
\begin{equation}
J(\Omega):=J(\Omega;u(\Omega))= \min_\Omega \frac{1}{2}\int_\Gamma u^2\ds
\end{equation}
where the function $u$ solves the boundary value problem
\begin{alignat}{3}\label{eq:poisson-a}
    -\Delta u &= f &\qquad & \text{in }\Omega \\
    \label{eq:poisson-b}
    u &= g_D & \qquad &\text{on }\Gammafix:=\partial\Omega\setminus\Gamma\\
    \label{eq:poisson-c}
    n \cdot\nabla u &= g_N& \qquad &\text{on }\Gamma
  \end{alignat}
Note that we keep the Neumann condition on the free boundary and enforce the Dirichlet condition through the minimization of the functional $J(\Omega)$.

The weak formulation of (\ref{eq:poisson-a}-{\ref{eq:poisson-c}) reads: 
find $u\in V_{g_D}(\Omega):=\{v\in
H^1(\Omega):v|_{\Gamma_\mathrm{fix}}=g_D\}$ such that
\begin{equation}
    \label{eq:primal} 
    a(\Omega; u,v) := \int_\Omega\nabla u\cdot\nabla v\dx = \int_\Omega f \dx=:
    F(\Omega; v)\qquad \forall v\in V_0(\Omega)
\end{equation}
Let $\mcO$ be the set of admissible domains; then the constrained minimization
problem reads: find $\Omega\in\mcO$ such that
\begin{align}\label{eq:min-a}
    J(\Omega) &= \min_{\Omega\in\mcO} J(\Omega;v)\quad
    \\ \label{eq:min-b}
    \text{for }u\in V_{g_D}(\Omega)\text{ s.t. } &a(\Omega; u,v) = F(\Omega;v)\quad \forall v\in V_0(\Omega)
  \end{align}
To solve the minimization problem (\ref{eq:min-a})-(\ref{eq:min-b}), 
we define the corresponding Lagrangian as
\begin{equation}
    \mathcal{L}(\omega,v,q):=  J(\omega;u) - a(\omega; v,q)+F(\omega; q)
\end{equation}
That is, we seek the domain $\Omega$ such that
\begin{equation}
    \Omega =\argmin\{\omega\in\mcO\mid \min_{v\in V_{g_D}(\omega)}\max_{q\in V_0(\omega)}\mathcal{L}(\omega,v,q)\}
\end{equation}


\subsection{Cut Finite Element Method}
We will use a cut finite element method to discretize the boundary value 
problem (\ref{eq:primal}). Before formulating the method we introduce 
some notation. 

\paragraph{The Mesh and Finite Element Space.}
Let $\Omega_0$ be a polygonal domain such that all 
admissible domains $\Omega\in\mcO$ are subsets of $\Omega_0$, i.e. 
$\Omega \subset \Omega_0$. Let $\mathcal{T}_{h,0}$ denote a family of 
quasiuniform triangulations of $\Omega_0$ with mesh parameter 
$h \in (0,h_0]$ and define the corresponding space of continuous 
piecewise linear polynomials 
\begin{equation}
\Vhzero = \{v \in H^1(\Omega_0):v|_T \in P_1(T), \quad 
\forall T\in \mathcal{T}_{h,0}\}
\end{equation}
Given $\Omega\in \mcO$ we define the active mesh 
\begin{equation}
\mcTh = \{T \in \mathcal{T}_{h,0} :
 \overline{T} \cap \overline{\Omega} \neq \emptyset \}
\end{equation}
the union of the active elements 
\begin{equation}
\Omega_h = \cup_{T \in \mcTh} T
\end{equation}
and the finite element space on the active mesh
\begin{equation}
   \VhOmega = \Vhzero|_{\Omega_h}
\end{equation}
Let also $\mcF_{h}$ denote the set of interior 
faces in $\mcTh$ such that at least one of its 
neighboring elements intersect the boundary 
$\partial \Omega$, 
\begin{equation}
  \mcFh = \{F: T^+_F\cap\partial \Omega 
  \neq \emptyset \text{ or }T^-_F\cap\partial \Omega\neq \emptyset \}
\end{equation}
where, $T^+_F$ and $T^-_F$ are the two elements sharing the face $F$. 
On a face $F$ we define the jump
\begin{equation}
  \llbracket v\rrbracket = v|_{T^+_F} - v|_{T^-_F}
\end{equation}
where $T^+_F$ is the element with the higher index.

\paragraph{The Method.} 
We define the forms
  \begin{align} \label{eq:Ah}
  A_h(\Omega;v,w) &= a_h(\Omega;v,w) + s_h(v,w)
\\
      a_h(\Omega;v,w)&=(\nabla v,\nabla w)_\Omega - (\partial_n v,w)_{\Gammafix}-(\partial_n w,v)_{\Gammafix} + (\gamma_D h^{-1} v,w)_{\Gammafix}  \notag \\
    s_h(\Omega;v,w) &= \sum_{F\in\mcFh} (\gamma_1 h \llbracket \partial_n v\rrbracket ,\llbracket \partial_n w\rrbracket)_{F} \label{eq:poisson_gradjump}
    \\
    F_h(\Omega;w) &= (f,w)_\Omega +(g_D,\gamma_D h^{-1}w - \partial_n w)_{\Gammafix}+ (g_N,w)_{\Gamma} \label{eq:poisson_rhs}
  \end{align}
were $(u,v)_\omega :=\int_{\omega} u\cdot v\, \text{d}\omega$ is the $L^2$ inner product over the set $\omega$ equipped with the appropriate measure. 
Our method for the approximation of  \eqref{eq:poisson-a}--\eqref{eq:poisson-c} takes the form: find 
$u_h \in \VhOmega$ such that
\begin{equation}\label{eq:fem}
   A_h(\Omega;u_h,v)  = F_h(\Omega;v)\quad 
  \forall v\in \VhOmega
\end{equation}
We recognize the weak enforcement of Dirichlet boundary conditions by Nitsche's method, cf. \cite{Ha05}. Furthermore, the term $s_h$, first suggested in this 
context in \cite{BH12}, is added to stabilize the method in the vicinity of the boundary.

\section{Shape Derivative} 
\label{shape_derivative}
\subsection{Definition of the Shape Derivative}
\label{sec:shape_derivative}
For $O\in \mcO$ 
we let $W(\Omega,\IR^d)$ denote the space of sufficiently smooth 
vector fields and for a vector field $\vartheta\in W$ we define the map
\begin{equation}
  \bndmap_\vartheta:\Omega \times I \ni (x,t)  \mapsto x + t \vartheta(x) \in \bndmap_\vartheta(\Omega,t) \subset \IR^d 
\end{equation}
where $I = (-\delta,\delta)$, $\delta>0$. For small enough $\delta$, 
the mapping $\Omega \mapsto \bndmap_\vartheta(\Omega,t)$ is a bijection and $\bndmap_\vartheta(\Omega,0) = \Omega$. We also assume that the vector field 
$\vartheta$ is such that $\bndmap_\vartheta(\Omega,t) \in \mcO$ for $t\in I$ with $\delta$ 
small enough.

Let $J(\Omega)$ be a shape functional, i.e., a mapping $J:\mcO\ni\Omega\mapsto J(\Omega) \in \IR$. We then have the 
composition 
$I \ni t \mapsto J \circ \bndmap(\Omega,t) \in \IR$ and we define the shape 
derivative $\DOV$ of $J$ in the direction $\vartheta$ by
\begin{equation}
    \DOV J(\Omega) 
    = 
    \frac{d}{dt} J\circ \bndmap_\vartheta(\Omega,t)|_{t=0}
    =
    \lim_{t \to 0}\frac{J(\bndmap_\vartheta(\Omega,t))-J(\Omega)}{t}
\end{equation}
Note that if  $\bndmap_\vartheta(\Omega,t) = \Omega$ we have  $\DOV J = 0$, 
even if $\bndmap_\vartheta$ change points in the interior of the domain.

We finally define the shape derivative 
$D_\Omega J|_\Omega: W(\Omega,\IR^d) \rightarrow \IR$ by
\begin{equation}
D_\Omega J|_\Omega(\vartheta) = \DOV J(\Omega)
\end{equation}
In cases when the functional $J$ depend on other arguments we use 
$\partial_\Omega$ to denote the partial derivative with respect to $\Omega$ 
and $\POV$ to denote the partial derivative with respect to 
$\Omega$ in the direction $\vartheta$.

\subsection{Leibniz Formulas}

For  $v:\Omega \times I  \rightarrow \IR^d$ we define the material time 
derivative in the direction $\vartheta$ by
  \begin{equation}
         \DtV v =  \lim_{t\to 0}\frac{v(\bndmap_\vartheta(x,t),t)-v(x,0)}{t}
  \end{equation}
and the partial time derivative by 
    \begin{equation}
        \partial_t v =\lim_{t\to 0}\frac{v(x,t)-v(x,0)}{t} 
    \end{equation}
From the chain rule it follows that
\begin{equation}
    \DtV v = \partial_t v   + \vartheta\cdot\nabla v
\end{equation}
The material derivative does not commute with the gradient and we have the 
commutator  
\begin{equation}
[\DtV, \nabla ] v 
= 
\DtV(\nabla v) 
-
 \nabla(\DtV v)
= 
 - (D\vartheta )^T\nabla v
\end{equation}
where $D\vartheta = V\otimes \nabla$ is the derivative (or Jacobian) of the vector 
field $\vartheta$, and the usual product rule
\begin{equation}
  \DtV{(vw)} = (\DtV v) w + v( \DtV w)  
\end{equation}
holds. To derive a expression of the shape derivative, the following lemma
will be used frequently.
\begin{lemma}\label{lem:leibniz}Let $f,g:\mathbb{R}^d\to\mathbb{R}$ be
functions smooth enough for the following expressions to be well defined. Then
the following relationships hold
\begin{align}
    \DOV\int_{\Omega} f\dx &= \int_\Omega \left(\DtV f + (\nabla\cdot \vartheta)f\right)\dx  \\
    \DOV\int_{\Gamma} g \ds
    &= \int_\Gamma \left(\DtV g + (\nabla_\Gamma\cdot \vartheta)g\right)\ds
\end{align}
where $\nabla_\Gamma\cdot v=\nabla\cdot v - n \cdot D\vartheta\cdot n$ 
and $n$ is the unit normal to $\Gamma$.
\end{lemma}
\begin{proof}
  The proof can e.g. be found in \cite{SZ92}.
\end{proof}
\begin{remark}
  An alternative definition of the surface divergence is
  $\nabla_\Gamma\cdot \vartheta = \text{tr}(\vartheta \otimes \nabla_\Gamma)$ where
  $\nabla_\Gamma v=(1-n \otimes n)\nabla v$ is the tangent gradient.
\end{remark}

\subsection{Shape Derivative of the Lagrangian formulation}
Recall the Lagrangian 
\begin{equation}
\mathcal{L}(\omega, v, q) = J(\omega; v) - a(\omega; v, q) + F(\omega; q)
\end{equation}
For a fixed $\Omega\in \mcO$ we take the Fr\'echet derivative 
$D\mcL(\Omega,\cdot,\cdot): V_0(\Omega) \times V_0(\Omega) \rightarrow \IR$
of $\mcL$,
\begin{align}
      D\mcL |_{(\Omega,u,p)} (\delta u,\delta p)
      = \left\langle D\mathcal{L}|_{(\Omega,u,p)}, \left( \delta u, \delta p\right)\right\rangle
      = \underbrace{\left\langle \partial_v \mcL|_{(\Omega,u,p)}, \delta u \right\rangle}_{=0}
      + 
      \underbrace{\left\langle \partial_q \mcL|_{(\Omega,u,p)}, \delta p \right\rangle}_{=0} = 0 
\end{align}
for any direction $\delta u, \delta p\in V_0(\Omega)$ and use the following identities
\begin{align}
\langle \partial_v \mcL|_{(\Omega,u,p)}, \delta u \rangle  
&= \langle \partial_v J(\Omega;u), \delta u \rangle - a(\Omega, \delta u, p )
= 0
\end{align}
which hold if $p$ solves the dual problem 
\begin{equation}  \label{eq:dual} 
	a(\Omega,v,p) =  m(v) := \langle \partial_v J(\Omega;u), v \rangle \qquad \forall v \in V_0(\Omega)
\end{equation}
where $\langle \partial_v J(\Omega;u), v \rangle=(u ,v)_{\Gamma}$ and 
\begin{align}
\langle \partial_q \mcL|_{(\Omega,u,p)}, \delta p \rangle 
&= F(\Omega, \delta p) - a(\Omega, u, \delta p ) 
=0
\end{align}
which holds since $u$ is the solution to the primal problem (\ref{eq:primal}).
The Correa-Seeger theorem \cite{CS85} states that
\begin{equation}
  \DOV \left(\min_{v\in V_{g_D}(\Omega)}\max_{q\in V_0(\Omega)}\mathcal{L}(\Omega,v,q)\right) = \POV \mathcal{L}(\Omega,u,p)
\end{equation}
and we obtain the shape derivative
\begin{equation}
  \DOV J(\Omega) = \POV \mathcal{L}(\Omega,u,p)
\end{equation}
\begin{lemma}\label{lem:shapeder}
  For $u$ and $p$ solving \eqref{eq:primal} and \eqref{eq:dual}, respectively, the shape derivative of
  $\mathcal{L}(\omega,v,q)$ in the point $(\Omega,u,p)$ is given  by
 	\begin{align}\label{eq:shape_derivative}
       \POV \mathcal{L}|_{(\Omega,u,p)} 
        &= - \int_\Omega \nabla u \cdot(D\vartheta +(D\vartheta )^T)\cdot \nabla p  \dx
        \\ \nonumber   
        &\qquad + \int_\Omega (\vartheta\cdot \nabla f) p  \dx
        \\ \nonumber
         &\qquad + \int_\Omega (\nabla \cdot \vartheta)\left( f - \nabla u\cdot  \nabla p \right)\dx
         \\\nonumber
        &\qquad +\int_{\Gamma}  (\nabla_\Gamma\cdot \vartheta)
        \left(2^{-1}u^2 + g_N p\right) \ds        
	\end{align}
\end{lemma}
\begin{proof} This is a well known result but we include the proof for the
convenience of the reader. We have
\begin{align}\label{eq:lem:shapeder-a}
\POV \mcL(\Omega,u,p) 
= \POV J(\Omega; u) 
- 
\POV  a(\Omega, u, p)
 + 
 \POV F(\Omega,p)
\end{align}
Using Lemma \ref{lem:leibniz} we obtain the identities
\begin{align*}
 \POV J(\Omega;u)
 =&{}
 \int_\Gamma \frac{1}{2}(\divs \vartheta) u^2\ds \\
\POV a(\Omega; u, p)  
= &{}\int_{\Omega} \left(\DtV (\nabla u \cdot \nabla p ) + (\divv \vartheta) \nabla u \cdot \nabla p\right)\dx
\\
= &{}\int_{\Omega} \left(\nabla u \cdot (D\vartheta + (D\vartheta)^T) \cdot \nabla p  
+ (\divv \vartheta) \nabla u \cdot \nabla p\right)\dx
\\ \nonumber
\POV F(\Omega;p) 
= &{} \int_{\Gamma}\left( \DtV(g_N p) 
+ (\divs \vartheta)(g_N p)\right)\ds 
+  \int_{\Omega} \left(\DtV(f p) 
+ (\divv \vartheta)(f p)\right)\dx
 \\
= & 
 \int_\Gamma (\divs \vartheta)(g_N p)\ds
 + \int_\Omega (\divv \vartheta)(f p)\dx
\end{align*}
Inserting these expressions into (\ref{eq:lem:shapeder-a}) we arrive at
\eqref{eq:shape_derivative}.

\end{proof}

\subsection{Finite Element Approximation of the Shape Derivative}

In order to compute an approximation of the shape derivatives 
we need pproximations of the solutions to the primal equation \eqref{eq:primal} and the dual equation \eqref{eq:dual}. We 
employ CutFEM formulations: find $u_h\in \VhOmega$  
such that
\begin{equation}\label{eq:u_h}
  A_h(\Omega;u_h,w) = F_h(\Omega;w)\qquad \forall w\in \VhOmega
\end{equation}
and $p_h\in \VhOmega$ such that 
\begin{equation}\label{eq:p_h}
  A_h(\Omega;p_h,w)  = m_h(w):=(u_h,w)_\Gamma\qquad 
  \forall w\in \VhOmega
\end{equation}
The discrete approximation of the shape derivative 
is obtained by inserting the discrete quantities $u_h,p_h$ into
\eqref{eq:shape_derivative}, i.e.,
  \begin{align} \label{eq:discrete_shapegradient}
    \DOV \mathcal{L}|_{(\Omega,u_h,p_h)}
    &= - \int_\Omega \nabla u_h \cdot(D\vartheta +(D\vartheta )^T)\cdot \nabla p_h  \dx
        \\ \nonumber   
        &\qquad + \int_\Omega (\vartheta\cdot \nabla f) p_h  \dx
        \\ \nonumber
         &\qquad + \int_\Omega (\nabla \cdot \vartheta)\left( f 
         - \nabla u_h \cdot  \nabla p_h \right)\dx
         \\\nonumber
        &\qquad +\int_{\Gamma}  (\nabla_\Gamma\cdot \vartheta)
        \left(2^{-1}u_h^2 + g_N p_h \right) \ds  
  \end{align}

\section{Velocity Field}
\label{velocity_field}
\subsection{Definition of the Velocity Field}
We follow \cite{Go06} to define a velocity field $\beta$ given the
shape derivative. We seek the velocity field $\beta$ such that we
obtain the largest decreasing direction of $D_{\Omega,\beta}\mcL|_{(\Omega,u,p)}$ 
under some regularity constraint, for instance assume that the velocity field 
is in $\Honed$ we obtain
\begin{equation}
  \label{eq:velocity_min}
	\beta = \argmin_{\|\vartheta\|_{\Honed}=1} D_{\Omega,\vartheta}\mcL|_{(\Omega,u,p)}
\end{equation}
Let $b$ be the $\Honed$ inner product 
\begin{equation}
	b(v,w):= \int_{\Omega_0} \left(D v : D w + v \cdot w \right)\dx 
\end{equation}
An equivalent formulation of the minimization problem 
\eqref{eq:velocity_min} is: find $\beta' \in \Honed$ such that
\begin{equation}
  \label{eq:velocity_galerkin} 
  b(\beta',\vartheta) = -D_{\Omega,\vartheta}\mcL|_{(\Omega,u,p)}
  \qquad\forall \vartheta\in  \Honed
\end{equation}
and set
\begin{equation}
  \label{eq:velocity_normalization} 
	\beta = \frac{\beta'}{\|\beta'\|_\Honed}
\end{equation}
It is then clear that $\beta$ is a descent direction since
\begin{equation}
 D_{\Omega,\beta} \mcL|_{(\Omega,u,p)} 
 = 
 - b(\beta',\beta) 
 = 
 - \|\beta'\|_{\Honed} \leq 0
\end{equation}

\begin{remark}\label{rem:velocity_field}
To prove the equivalence between the minimization problem \eqref{eq:velocity_min} 
and \eqref{eq:velocity_galerkin}--\eqref{eq:velocity_normalization} we compute 
the saddle point to the Lagrangian corresponding to \eqref{eq:velocity_min}. We 
obtain the Lagrangian
\begin{equation}
  \mathcal{K}(\tau,\lambda) = D_{\Omega,\tau}\mcL|_{(\Omega,u,p)} + \lambda\left( (\tau,\tau)_{\Honed} -1\right)
\end{equation} 
In the saddle point $(\tau,\lambda)$, we have that
\begin{equation}
  0 =\Big\langle\frac{\partial}{\partial \tau}\mathcal{K}(\tau,\lambda),\phi\Big\rangle = D_{\Omega,\phi}\mcL|_{(\Omega,u,p)} + 2\lambda(\tau,\phi)_{\Honed}
\end{equation}
and
\begin{equation}
  0 = \frac{\partial}{\partial \lambda}\mathcal{K}(\tau,\lambda) = (\tau,\tau)_{\Honed}  - 1
\end{equation}
holds. From \eqref{eq:velocity_galerkin} and \eqref{eq:velocity_normalization}
we see that $\beta'=2\lambda\tau$ where
\begin{equation}
  \beta =\frac{\beta'}{\|\beta'\|_{\Honed}}= \frac{2\lambda\tau}{\|2\lambda\tau\|_{\Honed}} = \tau
\end{equation}
and $\lambda = \|\beta'\|_{\Honed}/2$. Hence the two formulations are equivalent.
\end{remark}

\subsection{Regularity of the Velocity Field}
\label{sec:reg-vec-field}
Next we investigate the regularity of the velocity field $\beta$. For smooth
domains and under stronger regularity requirements the shape derivative can be
formulated using Hadamard's structure theorem as an integral over the
boundary,
\begin{equation}
  \DOV\mcL(\Omega,u,p) = (\mathcal{G},n\cdot \vartheta)_{L^2(\Gamma)} 
\end{equation}
where $\mathcal{G}$ is a function of the primal and
dual solutions $u$ and $p$, the right hand side, the boundary condition, and the mean
curvature. We thus note that the velocity field $\beta$ is a solution to the problem: 
find $\beta\in [H^1(\Omega_0)]^d$ such that
\begin{equation}\label{eq:weak-interface}
  b(\beta,\vartheta) = - (\mathcal{G},n\cdot \vartheta)_{L^2(\Gamma)}\qquad\forall \vartheta\in [H^1(\Omega_0)]^d
\end{equation}
The corresponding strong problem for each of the components $\beta_i$, $1=1,\dots,d$ 
of $\beta$ is
  \begin{alignat}{3}
    -\Delta \beta_i &= 0,&\qquad& \text{in } \Omega_0\setminus\Gamma  \\
    \beta_i &= 0,&\qquad& \text{on } \partial\Omega_0 \\
    [\beta_i] &= 0,&\qquad &\text{on }  \Gamma \\
    [n\cdot\nabla \beta_i] &= \mathcal{G}n_i,&\qquad& \text{on } \Gamma    
  \end{alignat}
which is an interface problem. Given that $\Gamma$ is smooth and 
$\mathcal{G} \in H^{1/2}(\Gamma)$, we have the regularity estimate
\begin{equation}\label{eq:regularity-interface}
  \|\beta_i\|_{H^1(\Omega_0)} 
  + \|\beta_i\|_{H^2(\Omega_0 \setminus \Gamma )}
  \lesssim \|\mathcal{G}\|_{H^{1/2}(\Gamma)}
\end{equation}
see \cite{CZ98}, and hence $\beta\in [H^1(\Omega_0)]^d\cap [H^2( \Omega_0\setminus\Gamma)]^d$. 


\subsection{Finite Element Approximation of the Velocity Field}

We define a discrete velocity field using a standard finite element 
discretization of (\ref{eq:velocity_galerkin}) with piecewise linear 
continuous trial and test functions $\Vhzero$ on $\Omega_0$. The 
discrete problem takes the form: find $\beta_h'\in [\Vhzero]^d$ such that
\begin{equation}
  \label{eq:discrete_velocity}
   b(\beta'_h,\vartheta) = -D_{\Omega,\vartheta}\mcL|_{(\Omega,u_h,p_h)}\qquad\forall\vartheta\in[\Vhzero]^d
 \end{equation}
 and set
\begin{equation}
  \label{eq:discrete_velocity_normalization} 
  	\beta_h = \frac{\beta'_h}{\|\beta'_h\|_\Honed}
\end{equation}

\section{Level Set Representation of the Free Boundary}
\label{level_set}
\subsection{Definition and Evolution of the Level Set Representation}
A level set function describing an interface needs to be evolved in order to find a
minimum to \eqref{eq:min-a}-\eqref{eq:min-b}. Let $\rho(x,\Gamma)$ be a distance
function defined as the minimal Euclidean distance between $x$ and $\Gamma$. The level set
function is the signed distance function
\begin{equation}
	\phi(x) = \begin{cases}\rho(x,\Gamma) &x\in\Omega_0\setminus\overline\Omega \\
						0           &x\in\Gamma \\
						-\rho(x,\Gamma)&x\in\Omega\end{cases}
\end{equation}
This function is moved by solving a Hamilton-Jacobi equation of the form
\begin{equation}\label{eq:levelset}
  \partial_t\phi + \beta\cdot \nabla\phi = 0
\end{equation}
After some time $\phi$ no longer resembles a discrete signed distance
function and so called reinitialization needs to be performed to restore the
distance properties. Reinitialization can be done by solving the Eikonal
equation
\begin{equation}\label{eq:Eikonal}
	\begin{cases}
		\partial_t \varphi + \mathrm{sign}(\phi)(|\nabla\varphi| - 1) = 0 & t \in(0,T]\\
		\varphi = \phi & t=0
	\end{cases}
\end{equation}
for the unknown $\varphi$. Setting $\phi = \lim_{T \rightarrow \infty}
\varphi(\cdot,T)$ yields a signed distance function on $\Omega$. In the
present paper we use a fast sweeping method to approximate \eqref{eq:Eikonal}
as suggested in \cite{DP12}.

\subsection{Finite Element Approximation of the Level Set Evolution}

To evolve the interface we use a standard finite element discretization 
of \eqref{eq:levelset}, using the space $\Vhzero$ of continuous piecewise 
linear elements on $\Omega_0$, with symmetric interior penalty stabilization, 
see \cite{BuFe09}, in space and a {Crank-Nicolson} scheme in time. Given 
a time $t_0$ we first determine a suitable time $T$ such that we may use 
$\beta_h(t_0)$ as an approximation of $\beta_h(t)$ on the interval 
$[t_0,t_0+T)$, then we divide $[t_0,t_0+T)$ into $N$ Crank-Nicolson steps 
of equal length. This procedure is repeated until a stopping criteria 
is satisfied. 

The time $T$ may be estimated using Taylor's formula
\begin{align} \nonumber
&\mcL(\bndmap_{\beta_h}(\Omega,\Delta t),u_h\circ\bndmap^{-1}_{\beta_h}(\Omega,\Delta t),p_h\circ \bndmap^{-1}_{\beta_h}(\Omega,\Delta t)) 
\\
&\qquad \approx \mcL(\Omega,u_h,p_h)|_{t=t_0} 
	+ D_{t,\beta_h}\mcL(\Omega,u_h,p_h)|_{t=t_0} \Delta t
\end{align}
Given a damping parameter $\alpha \in [0,1)$ we set $\mcL(\Omega_{t,\beta_h},u_h\circ \bndmap^{-1}_{\beta_h}(\Omega,t),p_h\circ \bndmap^{-1}_{\beta_h}(\Omega,t)) =
\alpha \mcL(\Omega,u_h,p_h)$ which yields the estimate
\begin{equation}\label{eq:T}
	T = \frac{(\alpha-1)\mcL(\Omega,u_h,p_h)}{D_{t,\beta}\mcL(\Omega,u_h,p_h)}
\end{equation}

To formulate the finite element method we divide $[t_0,t_0 + T)$ into 
$N$ time steps $[t_{n-1},t_n)$, of equal length $k=T/N$ and we use the notation 
$\phi_h^n = \phi_h(t_n)$ for the solution at time $t_n$. Given 
$\phi_h^0 \in \Vhzero$, find $\phi_h^n\in \Vhzero$ for $n=1,\dots,N,$ such that
\begin{equation}
  \label{eq:discrete_levelset}
  \begin{aligned}
  \left(\frac{\phi^n_h -\phi^{n-1}_h}{k}, w\right)_{\Omega_0} 
  + \left(\beta_h(t_0) \cdot\nabla\frac{\phi^n_h+\phi^{n-1}_h}{2},w\right)_{\Omega_0} 
  \\ 
    +r_h\left(\frac{\phi^n_h+\phi^{n-1}_h}{2},w\right) = 0 \qquad\forall w\in \Vhzero
  \end{aligned}
\end{equation}
where $r_h$ is the stabilization term 
\begin{equation}
     r_h(v,w) = \sum_{F \in \mcF_{h,0}} (\gamma_2 h^2[\partial_n v],[\partial_n w])_{L^2(F)} \label{eq:convection_stab}
 \end{equation}
where $\gamma_2>0$ is a parameter and $\mcF_{h,0}$ is the set of interior faces 
in the background mesh $\mcT_{h,0}$. 

\section{Optimization Algorithm}
\label{optimization_Algorithm}
In this section we summarize the optimization procedure and propose an
algorithm to solve \eqref{eq:min-a}-\eqref{eq:min-b}. During the optimization
procedure we use sensitivity analysis to compute the discrete shape derivate
\eqref{eq:discrete_shapegradient}, see Section~\ref{shape_derivative}. From
the discrete shape derivate we compute a velocity field $\beta_h$
\eqref{eq:discrete_velocity} using a $H^1(\Omega_0)$ regularization, which corresponds to the the greatest descent direction of the shape derivative in $H^1(\Omega_0)$, see
Section~\ref{velocity_field}. The velocity field is then used to move the level
set and update the free boundary, see Section~\ref{level_set}. These
steps are presented in Algorithm~\ref{alg:freeboundary}. As a stopping criterion we
require that the residual indicator
\begin{equation}
\label{eq:stop}
  R_\Gamma(u_h) = \|u_h\|_\Gamma  \leq \mathrm{TOL}
\end{equation}
for some tolerance $0<\mathrm{TOL}$.
\begin{algorithm}[h!tb]
  \caption{Bernoulli free boundary value problem}
  \label{alg:freeboundary}
  \begin{algorithmic}
    \State Input: A initial level set $\phi_h$, a damping parameter $\alpha$, and a
    tolerance $\mathrm{TOL}$.
    \State Compute primal solution $u_h$ \eqref{eq:u_h} and the residual indicator $R_\Gamma(u_h)$ 
    \eqref{eq:stop}
    \While{$R_\Gamma(u_h) >\mathrm{TOL}$}
      \State Compute the dual solution $p_h$ \eqref{eq:p_h}
      \State Compute the discrete shape derivative $\DOV\mcL_{(\Omega_h,u_h,p_h)}$
      \eqref{eq:discrete_shapegradient}
      \State Compute the velocity field $\beta_h$ \eqref{eq:discrete_velocity}
      \State Move the interface \eqref{eq:discrete_levelset}
      \State Compute the primal solution $u_h$ \eqref{eq:u_h}
      \State Compute the residual indicator $R_\Gamma(u_h)$ \eqref{eq:stop}
    \EndWhile %
  \end{algorithmic}
\end{algorithm}

\section{A Priori Error Estimates}
\label{a_priori}

In this section we derive an a priori error estimate for the velocity 
field in the $H^1(\Omega_0)$ norm. Recall that the regularity of the velocity field is given by (\ref{eq:regularity-interface}) and thus the best possible order 
of convergence is $O(h^{1/2})$ in the $H^1(\Omega_0)$ norm and $O(h^{3/2})$ 
in the $L_2(\Omega_0)$ norm, since we use a standard finite element method to approximate 
the velocity field. To prove the error estimate for the velocity field we will need 
bounds for the discretization error of the primal and dual solutions in $L^4$ norms 
since the right hand side of the problem (\ref{eq:velocity_galerkin}) defining 
the velocity field is the shape derivative functional (\ref{eq:shape_derivative}), 
which is a trilinear form, depending on the primal and dual solutions 
as well as the test function. For simplicity, we derive error estimates in $L^p$ 
norms for the primal and dual solutions using inverse bounds in combination with 
$L^2$ error estimates. These bounds are of course not of optimal order but, in the 
relevant case $d\leq 3$, they are sharp enough to establish optimal order bounds 
for the velocity field, given the restricted regularity of the velocity field. We 
employ the notation $a\lesssim b$ to abbreviate the inequality $a\leq Cb$ where the
constant $0\leq C$ is generic constant independent of the mesh size.

\subsection{The Energy Norm }

\paragraph{Definition of the Energy Norm.}	
For $2\leq p<\infty$ we define the energy norm
\begin{equation}\label{eq:Lpnorm}
\tn v \tn^p_{p,h}  = \| \nabla v \|^p_{L^p(\Omega)} 
+ h \| n\cdot \nabla v \|^p_{L^p(\Gammafix)}
+ h^{1-p}\| v \|^p_{L^p(\Gammafix)}
+ h \tn v \tn^p_{L^p(\mcFh)}
\end{equation}
where 
\begin{equation}\label{eq:Lpstab}
 \tn v \tn^p_{L^p(\mcFh)} = \sum_{F \in \mcFh} \|[n\cdot \nabla v]\|^p_{L^p(F)}
\end{equation}

\paragraph{An Inverse Estimate.}
We have the inverse estimate: for all $v \in \VhOmega$ it holds
\begin{equation} \label{eq:energyinverse}
\tn v \tn_{h,p} \lesssim h^{d(1/p - 1/2)} \tn v \tn_{h,2}
\end{equation}
To verify (\ref{eq:energyinverse}) we first note that using the 
inverse estimate
\begin{equation}
h^{1/p}\| n\cdot \nabla v\|_{L^p(F)}\lesssim \| \nabla v \|_{L^p(T)} 
\end{equation}
where $F$ is a face on the boundary of $T$, and the fact that 
the mesh $\mcTh$ (which we recall consists of full elements) 
covers $\Omega$ we have
\begin{equation}
\tn v \tn_{h,p} 
\lesssim 
\| \nabla v \|_{L^p(\mcTh)}  + h^{1/p-1} \|v\|_{L^p(\Gamma_{\text{fix}})}
\end{equation}
Next using the inverse estimates 
\begin{gather}
\| \nabla v \|_{L^p(T)} \lesssim h^{d(1/p - 1/2)} \| \nabla v \|_{T}
\\
h^{1/p-1} \|v\|_{L^p(F)}
\lesssim 
h^{1/p-1} h^{(d-1)(1/p - 1/2)} \| v \|_{F} 
\lesssim 
h^{d(1/p - 1/2)} h^{-1/2} \| v \|_{F} 
\end{gather}
to pass from  $L^p$ to $L^2$ norms we obtain
\begin{align}
\tn v \tn_{h,p}  
&\lesssim 
\| \nabla v \|_{L^p(\mcTh)} 
+ h^{1/p-1} \| v \|_{\Gamma_{\text{fix}}}
\\
&\lesssim h^{d(1/p - 1/2)} \Big( 
 \| \nabla v \|_{\mcTh} + h^{-1/2} \| v \|_{{\Gamma_{\text{fix}}}}
 \Big)
 \\
 &\lesssim h^{d(1/p - 1/2)} \tn v \tn_{h,2}
\end{align} 
where in the last step we used the estimate
\begin{equation}
\| \nabla v \|_{\mcTh} \lesssim \|\nabla v \|_\Omega + h^{1/2} \tn v \tn_{\mcFh}
\end{equation}
see \cite{BCHLM15}.

\subsection{Interpolation}

\paragraph{Definition of the Interpolation Operator.} 
We recall that there is an extension operator 
$E:W^s_p(\Omega) \rightarrow W^s_p(\Omega_\delta)$, 
for $0 \leq s$ and $1\leq p \leq \infty$, where 
$\Omega_\delta = \Omega \cup U_\delta(\Gamma)$ with $U_\delta(\Gamma)$ 
the tubular neighborhood $\{ x \in \IR^d : \rho(x,\Gamma)<\delta\}$. 
For $h \in (0,h_0]$, with $h_0$ small enough, we have 
$\Omega \subset \Omega_h \subset \Omega_\delta$. 
Let $\pi_{h}:L^1(\Omega_h) \to \VhOmega$ be a Scott-Zhang 
type interpolation operator, see \cite{ScZh90}, and for 
$u \in L^1(\Omega)$ we define $\pi_h v = \pi_h (E v)$.  For 
convenience we will use the simplified notation $v = E v$ 
on $\Omega_\delta$. 

\paragraph{Interpolation Error Estimates.}
We have the elementwise interpolation estimate 
	\begin{equation}\label{eq:interpol-element}
		h^{-1}\|v-\pi_h v\|_{L^p(T)} + \| \nabla (v-\pi_h v) \|_{L^p(T)} 
		\lesssim h\| v\|_{W^2_p(N(T))}
	\end{equation}
where $N(T)$ is the the set of neighboring elements in $\mcTh$ to 
element $T$. Summing over the elements and using the stability of the 
extension operator we obtain the interpolation error estimate
	\begin{equation}\label{eq:interpol}
		h^{-1}\|v-\pi_h v\|_{L^p(\Omega_h)} + \tn v -\pi_h v\tn_{p,h} 
		\lesssim h\| v \|_{W^2_p(\Omega_h)}
		\lesssim h\| v \|_{W^2_p(\Omega)}
	\end{equation}
We also have the following  interpolation error estimate in the energy norm
\begin{equation}\label{eq:interpol-energy}
\tn v - \pi_h v \tn_{h,p} \lesssim h \| v \|_{W^2_p(\Omega)}
\end{equation}

\paragraph{Verification of (\ref{eq:interpol-energy}).}
Using the element wise trace inequality
\begin{equation}
\| w \|^p_{L^p(F)} \lesssim h^{-1} \| w \|_{L^p(T)}^p 
+ h^{p-1} \|\nabla w \|^p_{L^p(T)}
\end{equation}
where $F$ is a face on $\partial T$, to estimate the terms on $\Gamma_{\text{fix}}$,
\begin{equation}
h \| n\cdot \nabla w \|^p_{L^p(\Gammafix)}
\lesssim 
\| \nabla w \|^p_{L^p(\mcTh(\Gammafix))} 
+ h^p \|\nabla \otimes \nabla w \|^p_{L^p(\mcTh(\Gammafix))} 
\end{equation}
\begin{equation}
h^{1-p}\| w \|^p_{L^p(\Gammafix)}
\lesssim 
h^{-p} \| w \|^p_{L^p(\mcTh(\Gammafix))} 
+ \| \nabla w \|^p_{L^p(\mcTh(\Gammafix))} 
\end{equation}
and the face stabilization term 
\begin{equation}
h \tn w \tn^p_{\mcFh} 
\lesssim \|\nabla w \|^p_{\mcTh(\mcFh)} 
+ h^p \|\nabla \otimes \nabla w \|^p_{\mcTh(\mcFh)} 
\end{equation}
We thus conclude that
\begin{align}
\tn w \tn_{h,p}^p 
&\lesssim 
\| \nabla w \|^p_\mcTh 
+ h^{-p}  \| w \|^p_{L^p(\mcTh(\Gammafix))} 
+ \| \nabla w \|^p_{L^p(\mcTh(\Gammafix))} 
+h^p \|\nabla \otimes \nabla w \|^p_{L^p(\mcTh(\Gammafix))} 
\\
\nonumber
&\qquad +  \|\nabla w \|^p_{\mcTh(\mcFh)} 
+ h^p \|\nabla \otimes \nabla w \|^p_{\mcTh(\mcFh)} 
\end{align}
Setting $w = v - \pi_h v$ and using the interpolation error 
estimate (\ref{eq:interpol-element}) and the identity 
$h^p \|\nabla \otimes \nabla (v - \pi_h v ) \|^p_{L^p(T)} 
= h^p \|\nabla \otimes \nabla v \|^p_{L^p(T)}$, which holds since 
we consider piecewise linear elements, we conclude that
\begin{equation}
\tn v - \pi_h v \tn_{h,p} \lesssim h \| v \|_{W^2_p(\Omega)}
\end{equation}

\subsection{Error Estimates for the Primal and Dual Solutions} 
\begin{lemma}\label{lem:primal}
The finite element 
approximation $u_h$ defined by (\ref{eq:fem}) of the solution $u$ to the
primal problem (\ref{eq:primal}) satisfies the a priori error estimate
	\begin{align} \label{eq:error_estimate_u}
		h^{-1}\|u-u_h\|_{L^p(\Omega)} + \tn u-u_h\tn_{p,h}
		\leq h^{1+d(1/p-1/2 )}\|u\|_{W^2_p(\Omega)}
	\end{align}
for $2\leq p < \infty$.
\end{lemma}
\begin{proof}
Using the triangle inequality we obtain
	\begin{align}
		\tn u- u_h \tn_{p,h} &\leq  \tn u-\pi_h u \tn_{p,h} + \tn \pi_h u- u_h \tn_{p,h}
\\
&\lesssim \label{eq:error-est-primal-a}
h \|u\|_{W^2_p(\Omega)} + \tn \pi_h u- u_h \tn_{p,h}
	\end{align}
where we employed the energy norm interpolation estimate  (\ref{eq:interpol}). 
For the second term on the right hand side of (\ref{eq:error-est-primal-a}) 
we employ the inverse inequality (\ref{eq:energyinverse})
with $v = \pi_h u - u_h$, 
	\begin{align}
		\tn \pi_hu - u_h\tn_{p,h} &\lesssim h^{d(1/p - 1/2)}\tn \pi_hu - u_h\tn_{2,h} 
		\\   \label{eq:error-est-primal-b}
		& \lesssim h^{d(1/p - 1/2)}\left(\tn u - \pi_h u\tn_{2,h}+\tn u - u_h\tn_{2,h}\right)
		\\  \label{eq:error-est-primal-c}
			& \lesssim h^{d(1/p - 1/2)} h\|u\|_{H^2(\Omega)}
\\   \label{eq:error-est-primal-d}
			& \lesssim h^{d(1/p - 1/2)} h\|u\|_{W^2_p(\Omega)}
	\end{align}
where in (\ref{eq:error-est-primal-b}) we added and subtracted $u$ 
and used the triangle inequality, in (\ref{eq:error-est-primal-c})  
we used the interpolation error estimate (\ref{eq:interpol-energy}) with 
$p=2$  together with the standard error estimate 
	\begin{equation}
		\tn u - u_h\tn_{2,h} \lesssim h\|u\|_{H^2(\Omega)}
	\end{equation}
	see \cite{BH12}, and in (\ref{eq:error-est-primal-d}) we 
	used the fact that $p>2$. 
	
	Finally, we estimate $h^{-1} \| u - u_h \|_{L^p(\Omega)}$ using 
	a standard duality argument. Let $\phi \in V_0(\Omega)$ be the solution 
	to the dual problem 
	\begin{equation}
	a(\Omega; v,\phi) = (\psi,v)\qquad \forall v \in V_0(\Omega)
	\end{equation}
    with $\psi \in L^q(\Omega)$ and $1/p + 1/q = 1$. Then we have the elliptic 
    regularity estimate $\| \phi \|_{W^2_q(\Omega)} \lesssim \|\psi\|_{L^q(\Omega)}$, 
    and using concistency we conclude that 
    \begin{equation}
    A_h(\Omega;v, \phi ) = (\psi,v)_\Omega \qquad \forall v \in \VhOmega + V_0(\Omega)
\end{equation}     
    Setting  $\psi  = (u - u_h)|u-u_h|^{p-2}$ and $v = u - u_h$ we obtain 
    \begin{align}
  \| u - u_h \|^p_{L^p(\Omega)} 
    &= a(u-u_h,\phi)
    \\
    &=A_h(\Omega; u-u_h,\phi) 
    \\
    &=
    A_h(\Omega; u-u_h,\phi - \pi_h \phi)
    \\
    &\leq \tn u - u_h \tn_{p,h} \tn \phi - \pi_h \phi \tn_{q,h}
    \\
    &\lesssim h \tn u - u_h \tn_{p,h} \| \phi \|_{W^2_q(\Omega)}
    \\ 
    &\lesssim h \tn u - u_h \tn_{p,h} \| u - u_h \|_{L_p(\Omega)}^{p/q}
\end{align}     
where we used the identity
$\| \psi \|_{L^q(\Omega)} = \| u-u_h \|_{L^q(\Omega)}^{p/q}$, and thus 
we conclude that 
\begin{equation}
\| u - u_h \|_{L^p(\Omega)}
\lesssim 
 h \tn u - u_h \tn_{p,h}
\end{equation}
since $p-p/q=1$.
\end{proof}
\begin{lemma} \label{lem:phapriori} The finite element approximation $p_h$ defined by 
(\ref{eq:p_h}) of the solution $p$ to the dual problem 
(\ref{eq:dual}) satisfies the a priori error estimate
    \begin{equation}\label{eq:error_estimate_p}
        \begin{aligned}
            h^{-1}\|p-p_h\|_{L^p(\Omega)} + \tn p-p_h\tn_{p,h} 
            & \lesssim  h^{1+d(1/p-1/2 )}\left(\|u\|_{W^2_p(\Omega)}+\|p\|_{W^2_p(\Omega)}\right)
        \end{aligned}
    \end{equation}
for $2 \leq p < \infty$.
\end{lemma}
\begin{proof}
We proceed as in the proof of Lemma~\ref{lem:primal}, with the difference that
we need to account for the error in the right hand side. We obtain
\begin{align}
	\tn \pi_h p - p_h\tn_{p,h}^2&\lesssim h^{2d(1/p-1/2)}\tn \pi_h p - p_h\tn^2_{2,h} \\
	& \lesssim h^{2d(1/p-1/2)} A_h(\Omega; \pi_h p - p_h,\pi_h p - p_h) 
	\\
	& \lesssim h^{2d(1/p-1/2)}\left(A_h(\Omega; \pi_h p - p,\pi_h p - p_h) + a(\Omega, p - p_h,\pi_h p - p_h) \right) \\
	& \lesssim h^{2d(1/p-1/2)}\Big(\tn\pi_h p - p\tn_{2,h}\tn\pi_h p - p_h\tn_{2,h} 
	\\ \notag
	&\qquad\qquad + |m(\pi_h p - p_h) - m_h(\pi_h p - p_h)| \Big) 
	\\
	& \lesssim h^{2d(1/p-1/2)}\Big(h|p|_{H^2(\Omega)}\tn\pi_h p - p_h\tn_{2,h}
    \\ \nonumber	
	&\qquad \qquad +h \|u\|_{H^2(\Omega)} \tn \pi_h p - p_h \tn_{2,h}\Big) 
	\\
	& \lesssim h^{1+d(1/p-1/2)}\Big( |p|_{H^2(\Omega)}+|u|_{H^2(\Omega)}\Big) \tn \pi_h p - p_h \tn_{p,h}
\end{align}
where we used a trace inequality and (\ref{eq:error_estimate_u}) to conclude that
\begin{equation}\label{eq:newLp}
	m(v) - m_h(v) 
	= (u-u_h,v)_\Gamma 
    = \| u - u_h \|_{H^1(\Omega)} \| v \|_{H^1(\Omega)}	
	\lesssim		
	h \| u \|_{H^2(\Omega)} \tn v \tn_{2,h} 
\end{equation}

\end{proof}
To bound $\| p - p_h \|_{L^p(\Omega)}$ we use a duality argument as in the proof 
of Lemma \ref{lem:primal}.
\begin{remark}
  Lemma \ref{lem:primal} and \ref{lem:phapriori} are suboptimal for $p>2$. Numerical
  test shows that the optimal error estimates, obtained by setting $d=0$ in
  the bounds (\ref{eq:error_estimate_u}) and  (\ref{eq:newLp}), hold
  for sufficiently smooth $u$ and $p$. 
\end{remark}
\begin{remark}\label{rem:domain_approximation} In the analysis we have for simplicity 
assumed that the boundary is exact. The discrete approximation of the boundary
may, however, be taken into account in the analysis using the techniques in
\cite{BuHaLaZa16}, under the assumption that the piecewise linear level set
representation of the boundary is second order accurate and that the
associated discrete normal is first order accurate. Such an analysis shows that 
the geometric error is of order $O(h^2)$ and thus of optimal order. 
\end{remark}

\subsection{Error Estimate for the Velocity Field} 
\label{sub:estimate_for_the_velocity_field}

\begin{theorem}
	\label{thm:a_priori_velocity}
  Let $d \leq 3$, $\beta$ be the solution to (\ref{eq:velocity_galerkin}), 
  and $\beta_h$ be the solution to (\ref{eq:discrete_velocity}), then
  \begin{equation}\label{eq:a_priori_velocity}
    \|\beta-\beta_h\|_{H^1(\Omega_0)}\leq M^{1/2} h^{1/2}
  \end{equation}
  where 
  \begin{equation}
  M = \| \beta\|^2_{H^2(\Omega_0
\setminus \Gamma)} + \|u\|_{W^2_4(\Omega)}^4 + \|p\|_{W^2_4(\Omega)}^4 
+ \|f\|_{L^4(\Omega)}^4 + \|g_N\|_{L^4(\Gamma)}^4
  \end{equation}
\end{theorem}
\begin{proof}
Adding and subtracting a Scott-Zhang interpolant $\pi_{h} \beta$ 
and using the weak formulations (\ref{eq:velocity_galerkin}) and (\ref{eq:discrete_velocity}) we obtain
\begin{align}
  \|\beta-\beta_h\|_{H^1(\Omega_0)}^2 &= 
  (\beta-\beta_h,\beta-\pi_{h}\beta)_{H^1(\Omega_0)} 
  + (\beta-\beta_h,\pi_{h}\beta-\beta_h)_{H^1(\Omega_0)} \\
  &= (\beta-\beta_h,\beta-\pi_{h}\beta)_{H^1(\Omega_0)} 
  + D_{\Omega,e_h}\mcL(\Omega,u,p)-D_{\Omega,e_h}\mcL(\Omega,u_h,p_h)
\end{align}
where $e_h = \pi_h \beta - \beta_h$. Estimating the right hand side we arrive 
at the bound
\begin{equation}\label{eq:thm-IplusII}
  \|\beta-\beta_h\|_{H^1(\Omega_0)}^2
  \lesssim 
  \underbrace{\|\beta-\pi_h \beta\|_{H^1(\Omega_0)}^2}_{I}
+
|\underbrace{D_{\Omega,e_h}\mcL(\Omega,u,p) 
-D_{\Omega,e_h}\mcL(\Omega,u_h,p_h)}_{II} |   
\end{equation}
Here Term $I$ is an interpolation error term which needs special treatment 
due to the limited regularity (\ref{eq:regularity-interface}) of $\beta$ across the interface $\Gamma$ and
Term $II$ accounts for the error in the velocity field that 
emanates from the approximation of the primal and dual solutions 
in the discrete problem (\ref{eq:discrete_velocity}). 

\paragraph{Term $\bfI$.}
Let $\mcT_{h,0}(\Gamma)$ 
be the set of all elements $T\in \mcT_{h,0}$ such that $N(T) \cap \Gamma \neq \emptyset$, 
where $N(T)$ is the set of all elements that are neighbors to $T$.  Then we 
have the estimates
\begin{equation}
\| \beta - \pi_h \beta \|_{H^1(T)} \lesssim h \|\beta \|_{H^2(N(T))}
\qquad T \in \mcT_h \setminus \mcT_h(\Gamma)
\end{equation}
and 
\begin{equation}
\| \beta - \pi_h \beta \|_{H^1(T)} \lesssim \|\beta \|_{H^1(N(T))}
\qquad T \in \mcT_h 
\end{equation}
see \cite{ScZh90}. Summing over all elements we obtain
\begin{align}
I&=\|\beta - \pi_h \beta\|_{H^1(\Omega_0)} ^2
\\
&\lesssim 
\sum_{T \in \mcT_{h,0} \setminus \mcT_{h,0}(\Gamma)}  h^2 \|\beta \|^2_{H^2(N(T))}
+ 
\sum_{T \in \mcT_{h,0} \setminus \mcT_{h,0}(\Gamma)}  \|\beta \|^2_{H^1(N(T))}
\\ \label{eq:thm-A-a}
&\lesssim h^2 \Big( \underbrace{\| \beta \|^2_{H^2(\Omega_1)} 
+  \| \beta \|^2_{H^2(\Omega_2)}}_{
=\| \beta \|^2_{H^2(\Omega\setminus \Gamma)}=:M_1} \Big)
+
\underbrace{\| \beta \|^2_{H^1(U_\delta(\Gamma))}}_{\bigstar} 
\end{align}
where $\Omega_1 = \phi^{-1}((-\infty,0])$, $\Omega_2 = \phi^{-1}([0,\infty))$, 
and $U_\delta(\Gamma) = \cup_{x\in \Gamma} B_\delta(x)$ is the 
tubular neighborhood of $\Gamma$ of thickness $\delta$. 

Observing that $\delta \sim h$ we may estimate $\bigstar$ by 
taking the $L^\infty$ norm in the direction orthogonal to 
$\Gamma$, in the following way
\begin{align}\label{eq:thm-A-aa}
\bigstar = \| \beta \|^2_{H^1(U_\delta(\Gamma))} 
&\lesssim 
h \sup_{t \in (-\delta,\delta)} \| \beta \|^2_{H^1(\Gamma_t)} 
\end{align}
where $\Gamma_t = \phi^{-1}(t)$. Next, we note that defining the domains
\begin{equation}
\Omega_{1,t} = \phi^{-1}((-\infty,t]),
\qquad 
\Omega_{2,t} = \phi^{-1}([t,\infty))
\end{equation}
we have
\begin{equation}
\Omega_{1,t} \subseteq \Omega_{1,0}=\Omega_1 \quad t \in (-\delta,0],
\qquad 
\Omega_{2,t} \subseteq \Omega_{2,0}=\Omega_2 \quad t \in [0,\delta)
\end{equation}
Therefore we have the trace inequalities
\begin{equation}
\| v \|_{H^1(\Gamma_t)} 
\lesssim \| v \|_{H^2(\Omega_{1,t})} 
\lesssim \| v \|_{H^2(\Omega_{1})}
\qquad t \in (-\delta,0] \qquad v \in H^1(\Omega_1)
\end{equation}
\begin{equation}
\| v \|_{H^1(\Gamma_t)} 
\lesssim \| v \|_{H^2(\Omega_{2,t})} 
\lesssim \| v \|_{H^2(\Omega_{2})}
\qquad t \in [0,\delta)\qquad v \in H^1(\Omega_2)
\end{equation}
which we may use to conclude that 
\begin{equation}
\sup_{t \in (-\delta,0]}  \| \beta \|^2_{H^1(\Gamma_t)} 
\lesssim \|\beta \|^2_{H^2(\Omega_{1})},\qquad 
\sup_{t \in [0,\delta)}  \| \beta \|^2_{H^1(\Gamma_t)} 
\lesssim \|\beta \|^2_{H^2(\Omega_{2})}
\end{equation} 
Thus we obtain the estimate
\begin{equation}
\sup_{t \in (-\delta,\delta)}  
\| \beta \|^2_{H^1(\Gamma_t)}
\lesssim 
\| \beta \|^2_{H^2(\Omega_{1})} 
+  \| \beta \|^2_{H^2(\Omega_{2})}
= M_1
\end{equation}
which together with (\ref{eq:thm-A-aa}) gives
\begin{equation}\label{eq:thm-A-b}
\bigstar \lesssim M_1 h
\end{equation}

Finally, combining estimates (\ref{eq:thm-A-a}) and (\ref{eq:thm-A-b}) we obtain
\begin{equation}\label{eq:refined-error-est-velocity}
I=\| \beta - \pi_h \beta \|^2_{H^1(\Omega_0)}
\lesssim 
M_1 (h +  h^2) 
\lesssim M_1 h
\end{equation}
for $0<h \leq h_0$. Similar bounds are used in \cite{BuHaLaZa16}. 

\paragraph{Term $\bfI\bfI$.} Using the notation 
$e_h = \pi_h \beta - \beta_h$ and $B = De_h+(De_h)^T$, we decompose 
Term $II$ as follows
  \begin{align}
      II&=D_{\Omega,e_h}\mcL(\Omega,u,p)-D_{\Omega,e_h}\mcL(\Omega,u_h,p_h)
      \\
      &= \int_\Omega ( B\nabla u \cdot \nabla p 
      - B\nabla u_h \cdot \nabla p_h ) \dx
\\ \nonumber      
      &\qquad +\int_\Omega(\nabla\cdot e_h)(\nabla u\cdot\nabla p-\nabla u_h\cdot\nabla p_h )\dx
      \\ \nonumber
      & \qquad 
      + \int_\Omega(\nabla\cdot e_h)f(p-p_h)\dx 
      \\ \nonumber 
      &\qquad 
      + \int_{\Gamma}(\nabla_\Gamma\cdot e_h)\frac{1}{2}(u^2 -u_h^2) \ds 
      \\ \nonumber
      &\qquad + \int_{\Gamma}(\nabla_\Gamma\cdot e_h) g_N(p-p_h)\ds \notag
      \\
      &=II_1 + II_2 + II_3 +II_4 + II_5
  \end{align}
  Utilizing the a priori error estimates (\ref{eq:error_estimate_u}) and (\ref{eq:error_estimate_p}) for the primal and dual problems we obtain 
  the following bounds 
     \begin{align}\label{eq:TermI1}
            II_{1} &\lesssim \delta \|e_h\|_{H^1(\Omega_0)}^2 
            + \delta^{-1} (h^2 + h^{4-d}) \Big( \|u\|_{W^2_4(\Omega)}^ 4 
            + \|p\|_{W^2_4(\Omega)}^4 \Big) 
            \\ \label{eq:TermI2}
            II_2 &\lesssim \delta\|e_h\|_{H^1(\Omega_0)}^2 
            + \delta^{-1} (h^2 + h^{4-d}) \Big( \|u\|_{W^2_4(\Omega)}^ 4 
            + \|p\|_{W^2_4(\Omega)}^4 \Big) 
            \\ \label{eq:TermI3}
            II_3&\lesssim \delta\|e_h\|_{H^1(\Omega_0)}^2
            + \delta^{-1}
            h^{(8-d)/2}\Big( \|f\|_{L^4(\Omega)}^4 + 
            \|p\|_{W^2_4(\Omega)}^4 \Big)  
            \\ \label{eq:TermI4}
            II_4 &\lesssim  \delta\|e_h\|^2_{H^1(\Omega_0)} 
            + \delta^{-1} h^{(5-d)/2}\|u\|^4_{W^2_4(\Omega)} 
             \\ \label{eq:TermI5}
            II_5 &\lesssim \delta\|e_h\|^2_{H^1(\Omega_0)} 
            + \delta^{-1} h^{(5-d)/2}\Big( \|g_N \|^4_{L^4(\Gamma)}
            + \|u\|^4_{W^2_4(\Omega)} \Big) 
     \end{align}
Detailed derivations of estimates (\ref{eq:TermI1})
-(\ref{eq:TermI5}) are included in Appendix A. Collecting the estimates 
(\ref{eq:TermI1})-(\ref{eq:TermI5}), using the fact $d \leq 3$, 
and defining 
\begin{align}
M_2 &= \|u\|_{W^2_4(\Omega)}^4 + \|p\|_{W^2_4(\Omega)}^4 
+ \|f\|_{L^4(\Omega)}^4 + \|g_N\|_{L^4(\Gamma)}^4
\end{align}
we obtain the bound
\begin{align} \label{eq:thmTermI-a}
II &\lesssim \delta\|e_h\|^2_{H^1(\Omega_0)} + \delta^{-1} M_2 h  
\\ \label{eq:thmTermI}
&\lesssim \delta\|\beta - \beta_h\|^2_{H^1(\Omega_0)} 
+ \delta \underbrace{\|\beta - \pi_h \beta\|^2_{H^1(\Omega_0)}}_{=I\lesssim M_1 h \text{ (\ref{eq:refined-error-est-velocity})}}+  \delta^{-1} M_2  h
\end{align}
where we added and subtracted $\beta$ 
in the first term. 

\paragraph{Conclusion of the Proof.} Starting from (\ref{eq:thm-IplusII}) 
and using the estimates (\ref{eq:refined-error-est-velocity}) and (\ref{eq:thmTermI}) of $I$ and $II$ we obtain
\begin{align}
\| \beta - \beta_h \|^2_{H^1(\Omega_0)} 
&\lesssim 
 \delta \|\beta - \beta_h \|^2_{H^1(\Omega_0)}
+ (1 + \delta ) M_1 h  + M_2 \delta^{-1} h 
\end{align}
and thus, taking $\delta = 1/2$, we obtain
\begin{align}
\| \beta - \beta_h \|^2_{H^1(\Omega_0)} 
&\lesssim 
(M_1 + M_2) h = M h
\end{align}
where $M=M_1 + M_2$, which completes the proof.
\end{proof}

\section{Numerical Examples}
\label{numerical_examples}
\subsection{Model Problems}
We use the following settings in the numerical examples
\begin{itemize}
	\item Optimization algorithm
	\begin{itemize}
		\item $\alpha = 0.5$: Damping parameter in \eqref{eq:T}	
		\item $N=3$: Number of time steps in \eqref{eq:discrete_levelset}
		\item $\mathrm{TOL}=10^{-5}$: Tolerance in \eqref{eq:stop}
	\end{itemize}
	\item Finite element methods
	\begin{itemize}
		\item $\gamma_1=1$: Penalty parameter for the gradient jump \eqref{eq:poisson_gradjump} 
		\item $\gamma_D=10$: Penalty parameter for the Dirichlet boundary condition \eqref{eq:poisson_rhs}
		\item $\gamma_2=1$: Penalty parameter for the gradient jump \eqref{eq:convection_stab}
	\end{itemize}
\end{itemize}

\paragraph{Model Problem 1.} To define the domain $\Omega$ we 
let  $\Omega_0 = [0,1]^2$ be the unit square and $\Omega_1 \subset \Omega_0$ be a domain in the interior of $\Omega_0$ with boundary $\Gamma$, finally, let $\Omega = \Omega_0 \setminus \Omega_1$. 
We note that $\partial \Omega = \Gamma \cup \partial \Omega_0$ and 
that $\Gamma \cap \partial \Omega_0 = \emptyset$.

With this set up we consider a Bernoulli free boundary value problem 
where the exact position of the free boundary $\Gamma$ is a circle of 
radius $r=0.25$ centered in $(0.5,0.5)$ and the exact solution is $u_{\mathrm{ref}}=4((x-0.5)^2 +
(y-0.5)^2)^{1/2}-1$. The corresponding Bernoulli free boundary problem takes 
the form
\begin{alignat}{3}
    -\Delta u &= -\Delta u_{\mathrm{ref}}&\qquad & \text{in }\Omega \\
    u &= u_{\mathrm{ref}} & \qquad & \text{on }\partial\Omega_0 \\
    n \cdot\nabla u &= -4 &\qquad & \text{on }\Gamma \\
    u &= 0 &\qquad & \text{on }\Gamma 
\end{alignat}
We will use a level set function corresponding to the domain displayed in
Figure~\ref{fig:model_start1} (right {sub-figure}) as an initial guess.
\begin{figure}[!htb]
  \centering
  \includegraphics[width=0.33\textwidth]{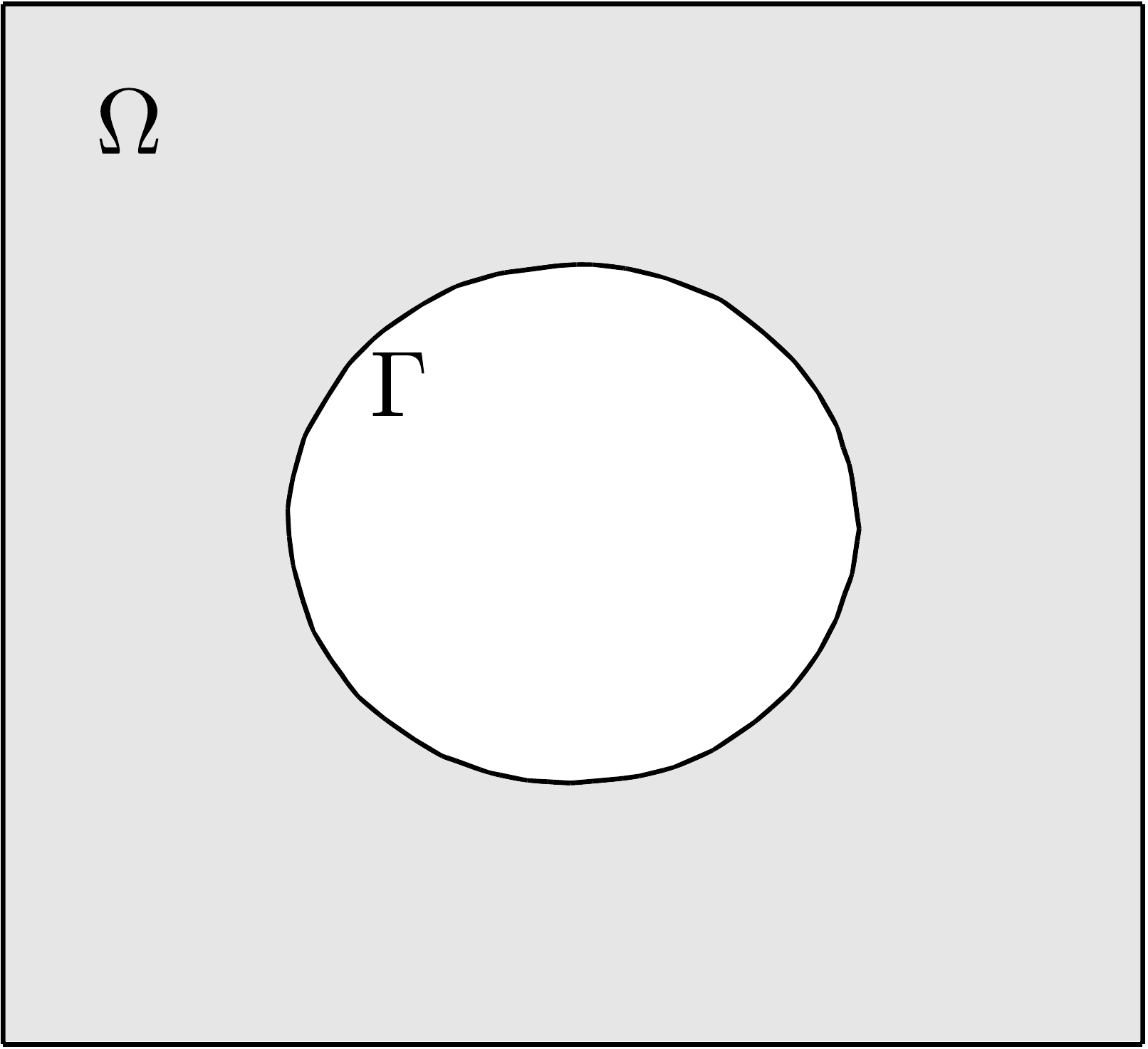}
  \includegraphics[width=0.33\textwidth]{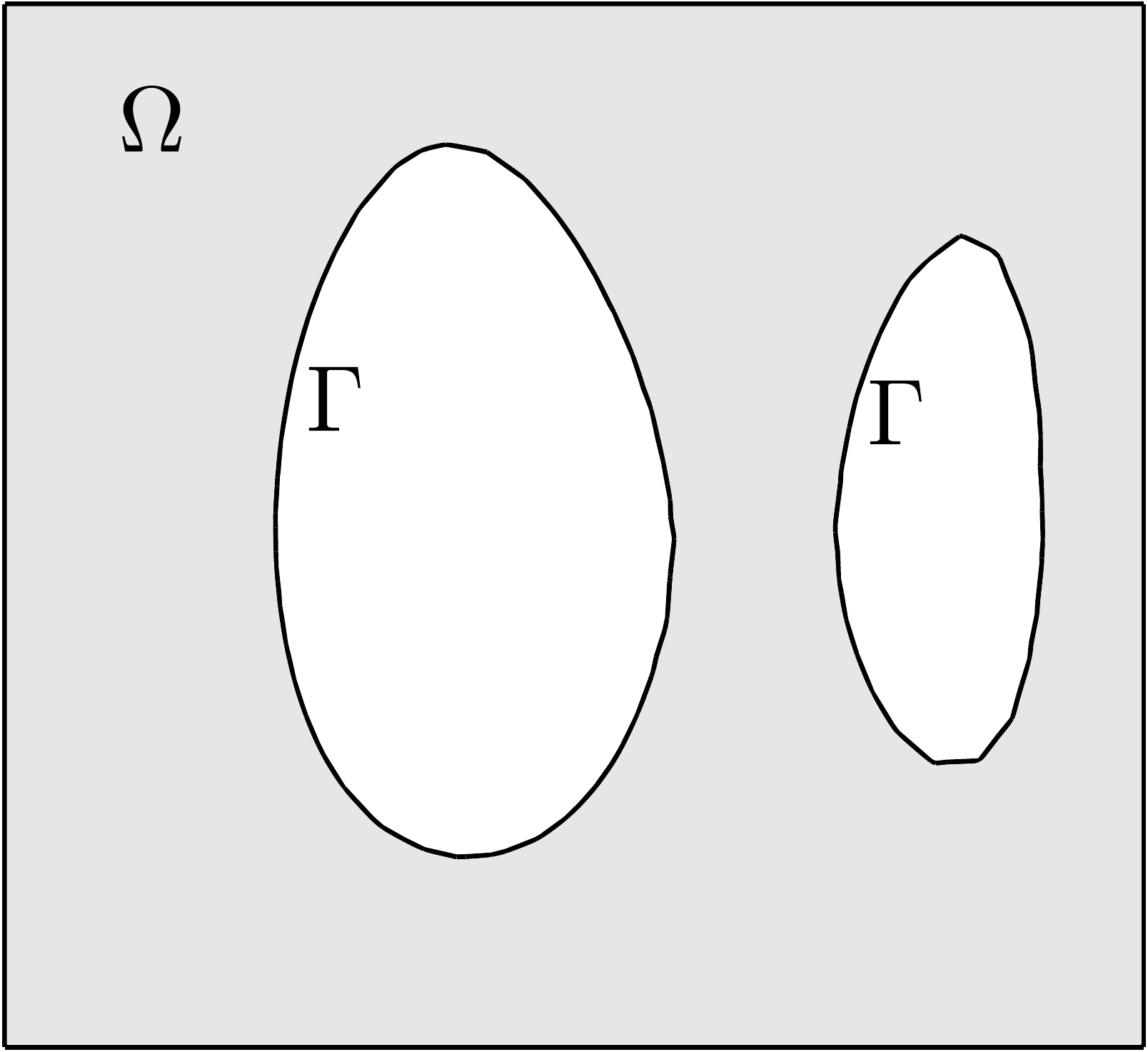}
  \caption{The final domain (left) and initial guess (right) for {Model Problem 1}. The gray area is the computational domain $\Omega$, the outer square boundary 
is fixed, and $\Gamma$ is the free boundary.}
  \label{fig:model_start1}    
\end{figure} 

\paragraph{Model Problem 2.}  Let $\Omega_0 = [0,1]^2$ be the unit square 
as before and $\Omega_1 \subset \Omega_0$ be a subdomain in the interior 
of $\Omega_0$ with boundary $\Gamma$. Next let 
$\Omega_2$ and $\Omega_3$ be the balls of radius $R=1/12$ centered in the points $(1/3,2/3)$ and $(2/3,1/3)$. Finally, set $\Omega= \Omega_1 \setminus (\Omega_2 \cup \Omega_3)$ and consider the boundary conditions
\begin{alignat}{3}
    u &=1 & \qquad & \text{on }\partial\Omega\setminus \Gamma \\
    n \cdot\nabla u &= -3 &\qquad & \text{on }\Gamma \\
    u &= 0 &\qquad & \text{on }\Gamma 
\end{alignat}
In this example there is no known exact position of the free boundary. We will
use a level set function corresponding to the domain
Figure~\ref{fig:model_start2} (right {sub-figure}) as an initial guess.
\begin{figure}[!htb]
  \centering
  \includegraphics[width=0.33\textwidth]{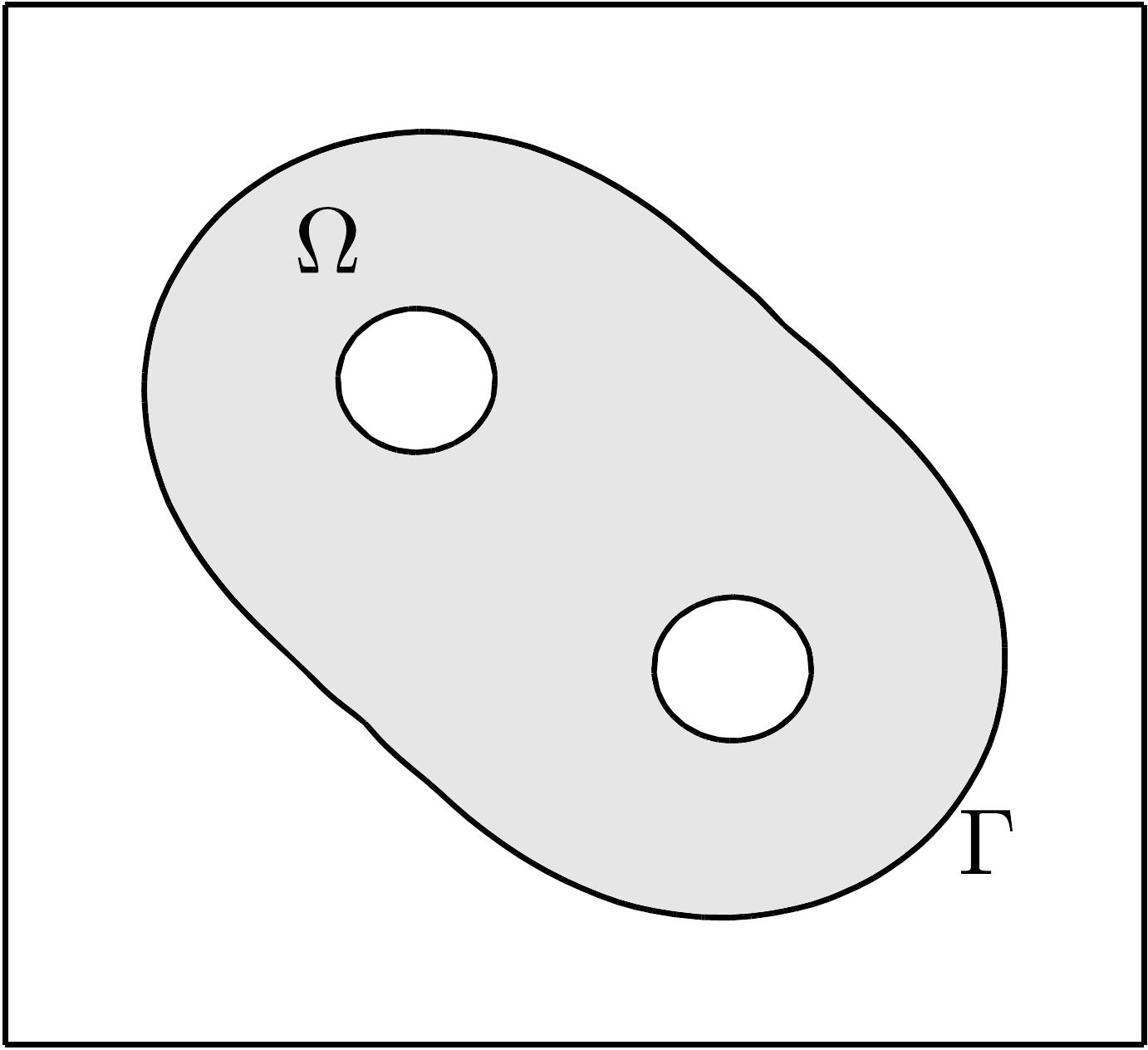}
  \includegraphics[width=0.33\textwidth]{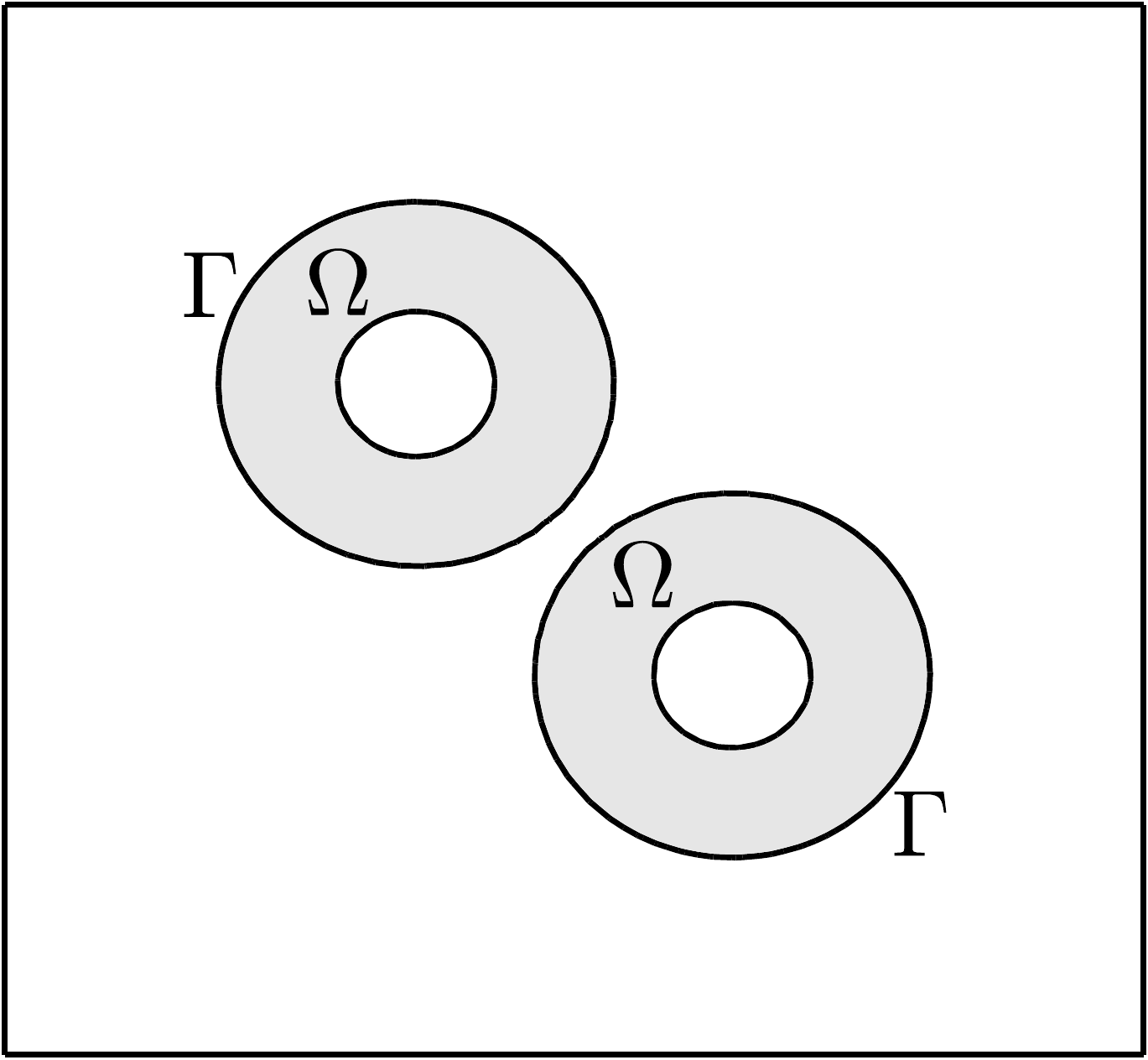}
  \caption{The final domain (left) and initial guess (right) for {Model Problem 2}. The gray area is the computational domain $\Omega$, the outer square boundary and the two inner most circles are fixed, and $\Gamma$ is the free bondary.}
  \label{fig:model_start2}    
\end{figure} 

\subsection{Convergence of the Velocity Field} 
We investigate the convergence rate of the discrete velocity field for {
Model Problem 1}. In Figure \ref{fig:model1_convergence} we display the error
in the discrete velocity field in the $H^1$-norm and $L^2$-norm where the
reference solution $\beta_{\mathrm{ref}}$ is computed  on a quasi-uniform mesh
with $526338$ degrees of freedom. We obtain slightly better convergence rates
than $\mathcal{O}(h^{1/2})$ and $\mathcal{O}(h^{3/2})$ in $H^1$ and
$L^2$-norm, respectively, which is in agreement with
Theorem~\ref{thm:a_priori_velocity}. 
\begin{figure}[!htb]
  \centering
  \includegraphics[width=0.5\textwidth]{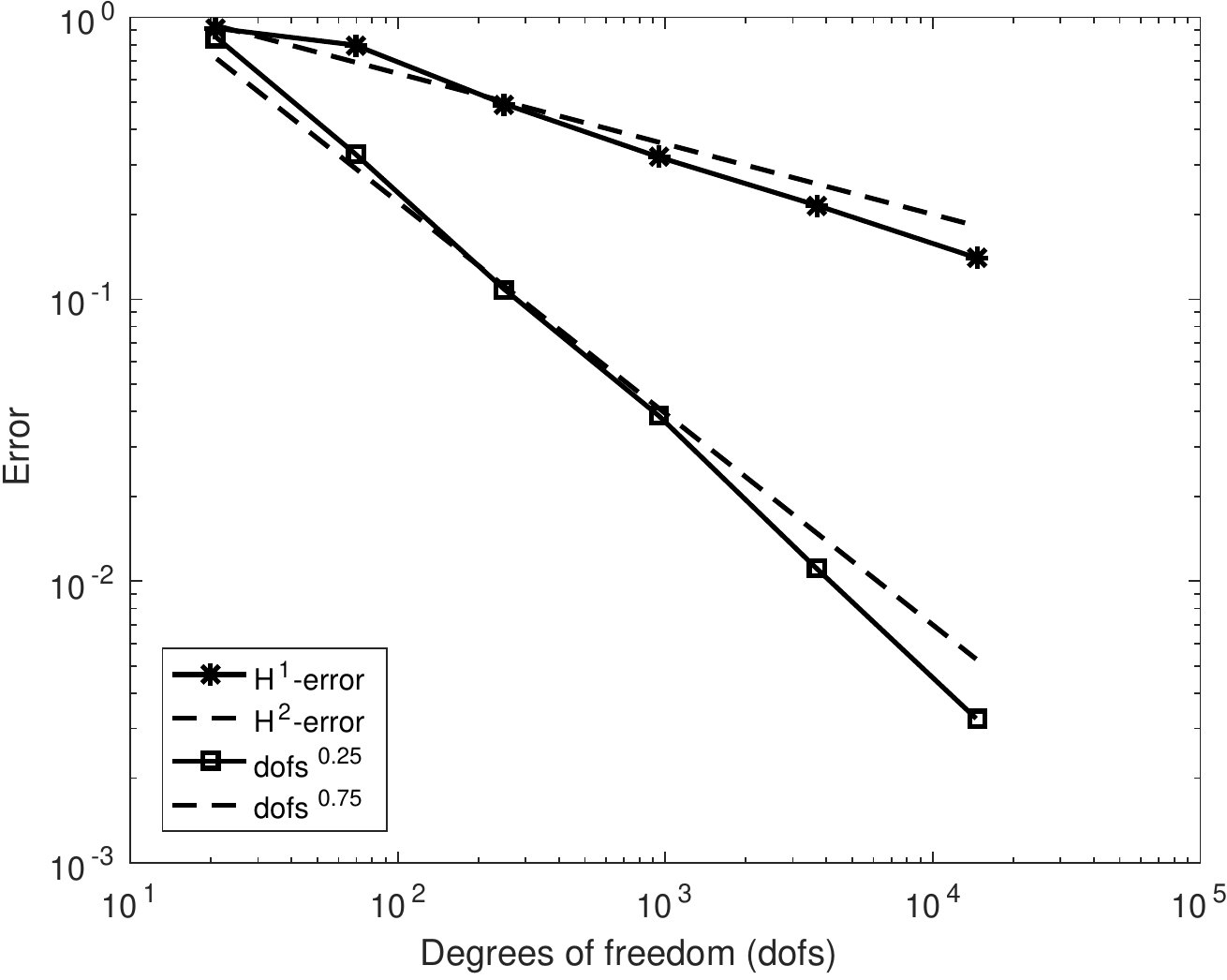}
  \caption{Convergence of the error in velocity field in Model Problem 1 
  in the $H^1$ and  $L^2$-norm.}
  \label{fig:model1_convergence}    
\end{figure}

\subsection{Free Boundary Problem}

In Figure~\ref{fig:model1_freeboundary} and Figure~\ref{fig:model2_freeboundary} we present the convergence history 
of $R_\Gamma$, see (\ref{eq:stop}), for Model Problem 1 and 2. 
In Figure~\ref{fig:model2_domain} we show the approximation of $\Omega$ obtained after $0,5,15,$ and $46$ iterations, where iteration $46$ is the final domain. In Figure~\ref{fig:model2_domain} we note that we rapidly obtain a domain 
which resembles the final domain, but to straighten the kinks in 
the boundary takes some extra effort.

\begin{figure}[!htb]
  \centering
  \includegraphics[width=0.5\textwidth]{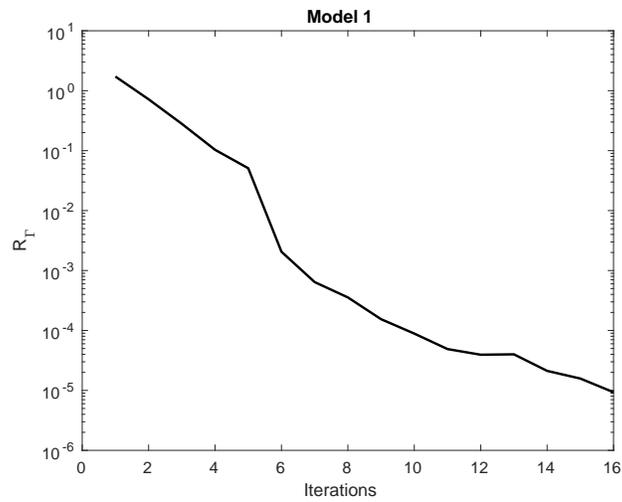}
\psfrag{Error}{$R_\Gamma$}  
  \caption{Convergence of $R_\Gamma$ for Model Problem 1.}
  \label{fig:model1_freeboundary}    
\end{figure}
\begin{figure}[!htb]
  \centering
  \includegraphics[width=0.5\textwidth]{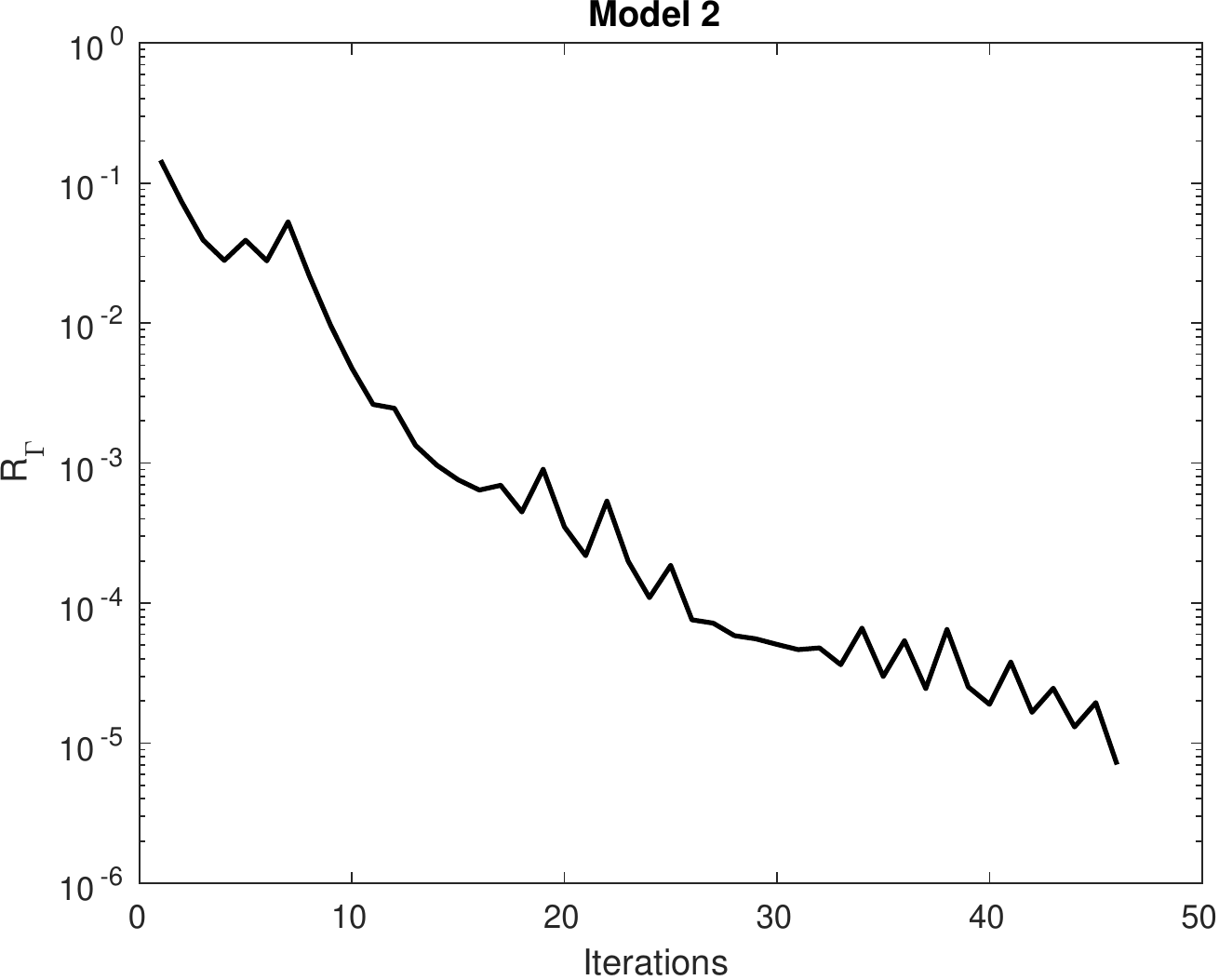}
  \caption{Convergence of $R_\Gamma$ for Model Problem 2.}
  \label{fig:model2_freeboundary}    
\end{figure}
\begin{figure}[!htb]
\begin{subfigure}{.5\textwidth}
  \centering
	\includegraphics[width=0.6\linewidth]{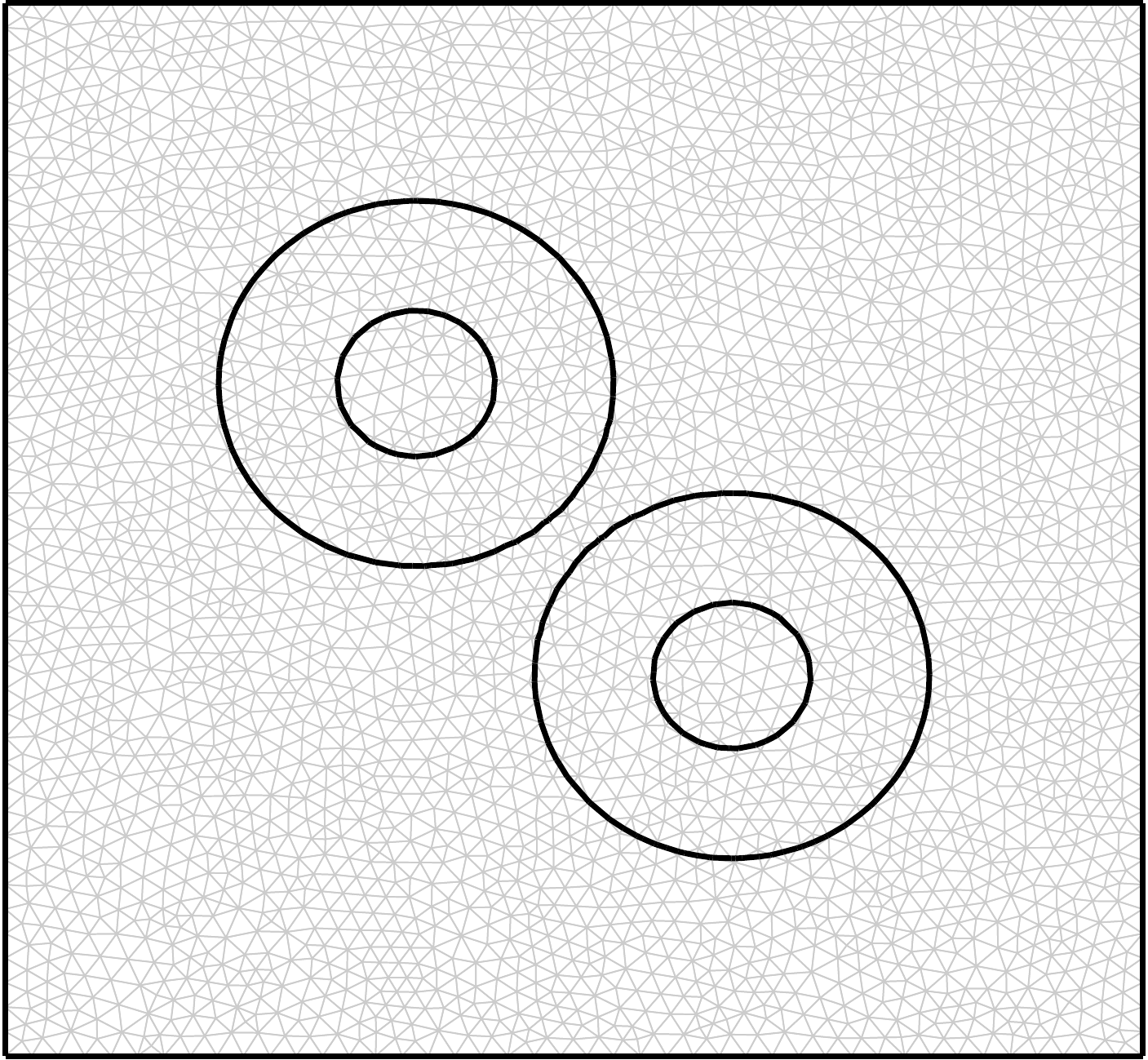}
  \caption{Iteration 0.}
\end{subfigure}
\begin{subfigure}{.5\textwidth}
  \centering
	\includegraphics[width=0.6\linewidth]{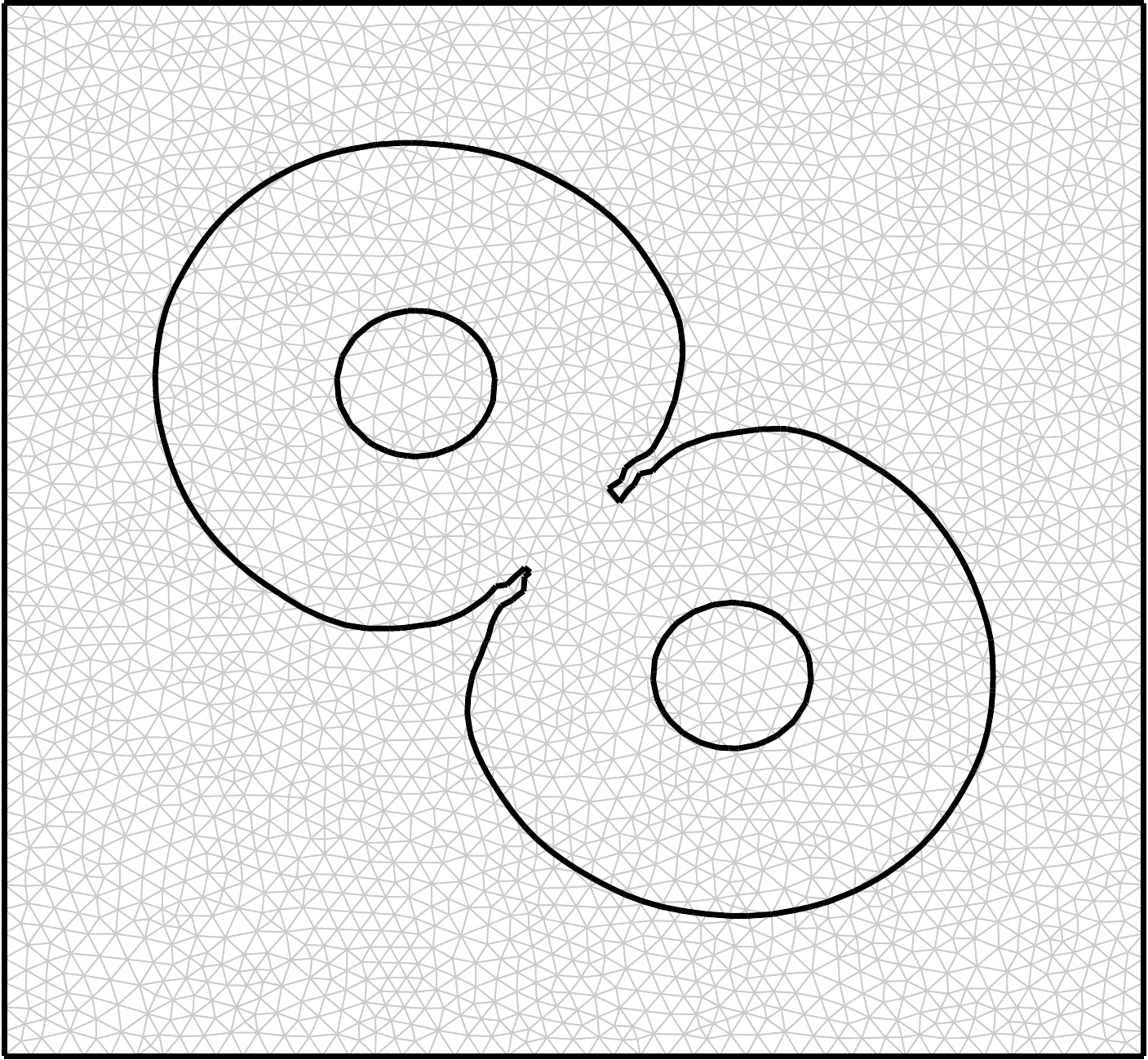}
  \caption{Iteration 5.}
\end{subfigure}
\begin{subfigure}{.5\textwidth}
  \centering
	\includegraphics[width=0.6\linewidth]{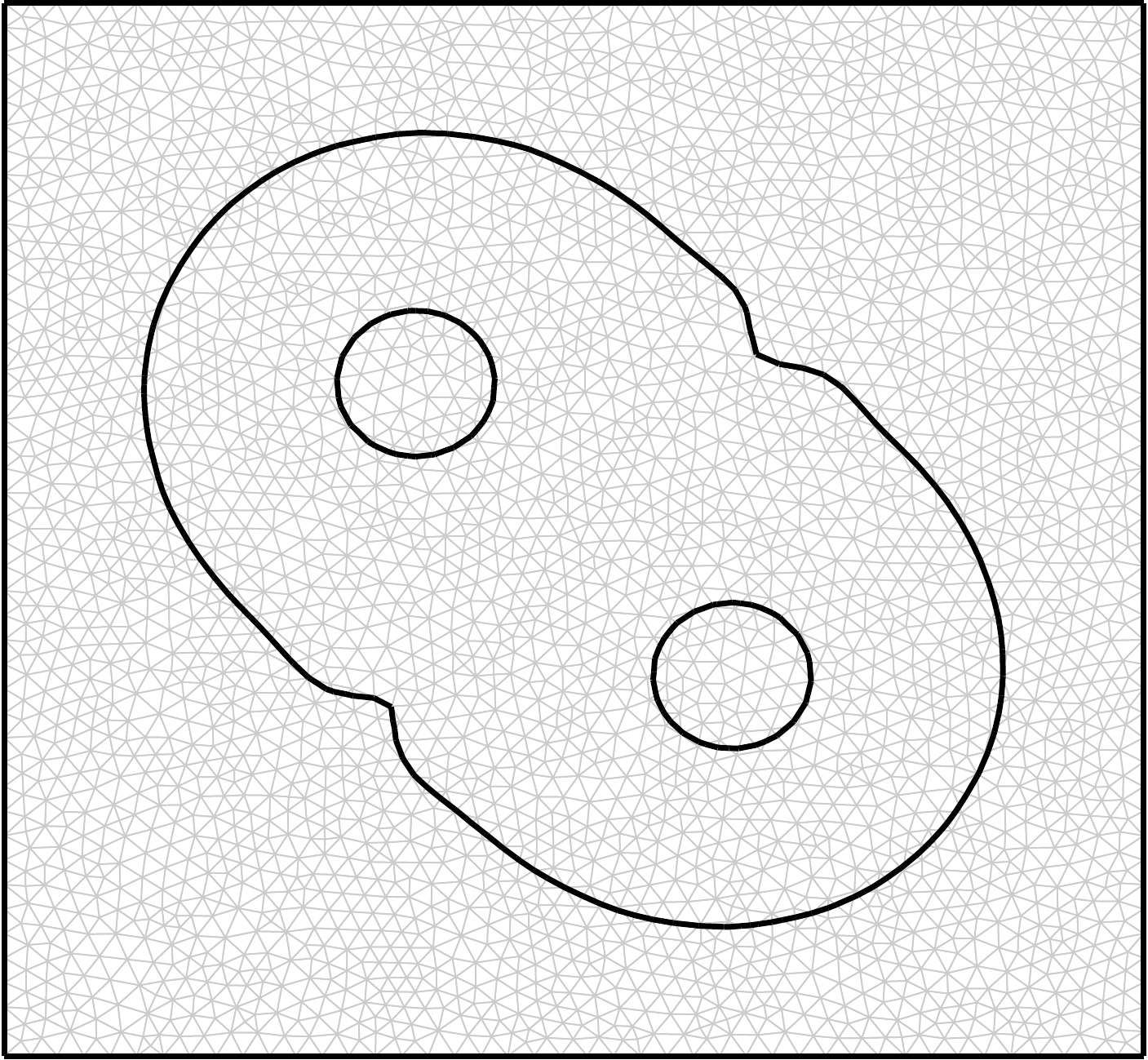}
  \caption{Iteration 15.}
\end{subfigure}
\begin{subfigure}{.5\textwidth}
  \centering
	\includegraphics[width=0.6\linewidth]{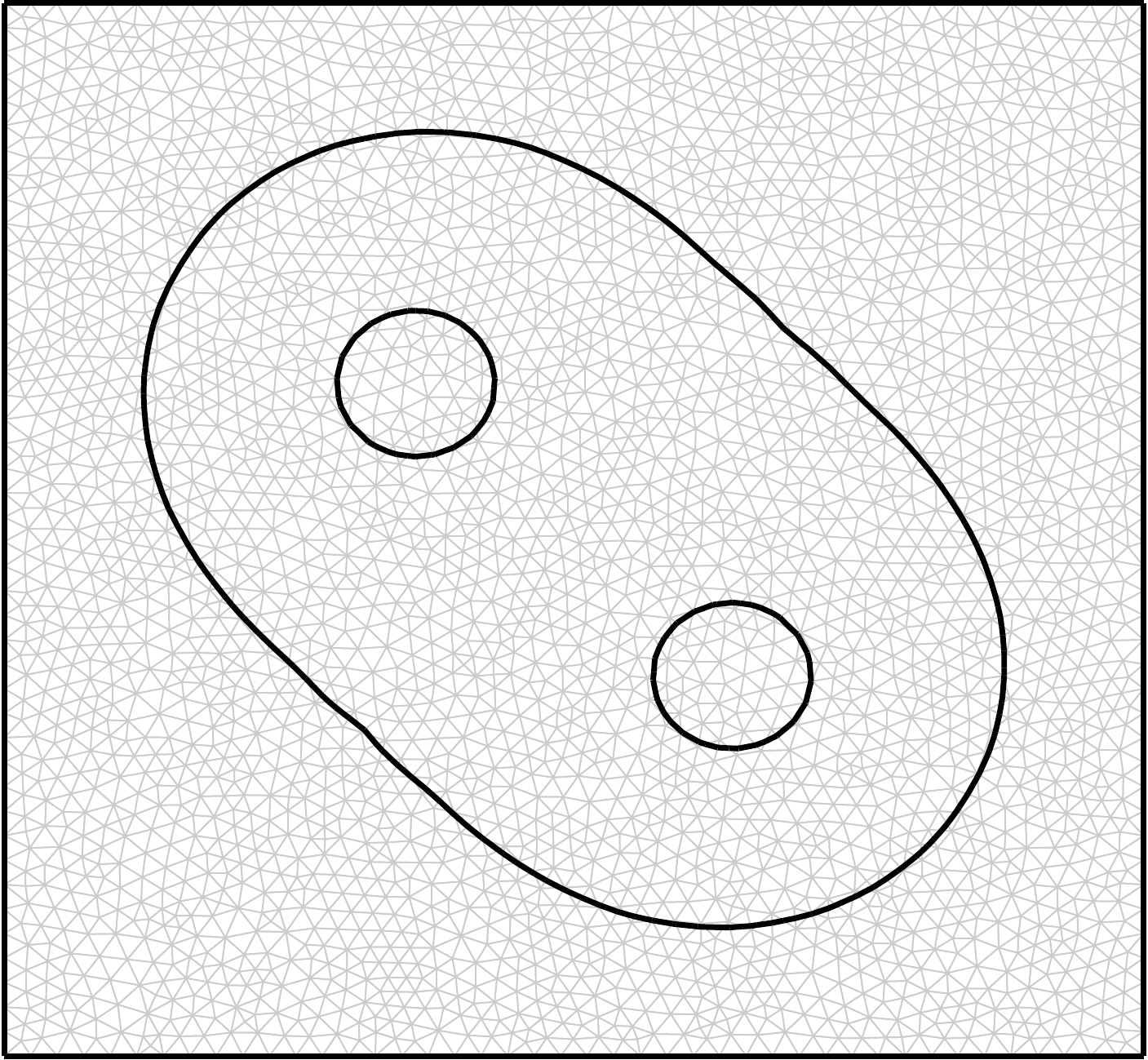}
  \caption{Iteration 46.}
\end{subfigure}
  \caption{The computational domain for {Model Problem 2} after 0, 5, 15, and 46
  iterations.}
  \label{fig:model2_domain}    
\end{figure} 

\newpage
\mbox{ } 
\newpage
\appendix 
\section{Bounds for $\boldsymbol{II_1- II_5}$ in the Proof of Theorem~\ref{thm:a_priori_velocity}}

\paragraph{Term $\boldsymbol{II_1}$.} 
Dividing $II_1$ into three suitable 
terms and then using H\"older's inequality we obtain\newpage
    \begin{align}
        II_1&= \int_{\Omega} ( B\nabla u \cdot \nabla p - B\nabla u_h \cdot \nabla p_h )\dx \\
        &=(B\nabla (u-u_h),\nabla p)_{L^2(\Omega)}
        \\ \nonumber
        &\qquad
        +(B\nabla u,\nabla (p-p_h))_{L^2(\Omega)}
        \\ \nonumber
        &\qquad-(B\nabla (u-u_h),\nabla (p-p_h))_{L^2(\Omega)}
     \\
     &\leq
     \|B\|_{L^2(\Omega)} \|\nabla (u-u_h)\|_{L^2(\Omega)} 
     \|\nabla p\|_{L^\infty(\Omega)}
     \\ 
     &\qquad \nonumber 
        + \|B\|_{L^2(\Omega)} \|\nabla u \|_{L^\infty(\Omega)}  
        \|\nabla (p-p_h)\|_{L^2(\Omega)}
     \\
     &\qquad \nonumber
        + \|B\|_{L^2(\Omega)} \|\nabla (u-u_h)\|_{L^4(\Omega)} 
            \|\nabla (p-p_h)\|_{L^4(\Omega)} 
        \\ \label{appendix:I1-b}
      &\leq
     \|e_h\|_{H^1(\Omega)} h\|u\|_{W^2_2(\Omega)} 
     \|p\|_{W^2_4(\Omega)}
     \\ 
     &\qquad \nonumber 
        + \|e_h\|_{H^1(\Omega)} \| u \|_{W^2_4(\Omega)}  
        h \| p \|_{W^2_2(\Omega)}
     \\
     &\qquad \nonumber
        + \|e_h\|_{H^1(\Omega)} h^{(4-d)/2}\| u \|_{W^2_4(\Omega)} 
            \| p \|_{W^2_4(\Omega)} 
     \\     \label{appendix:I1-c}
     &\lesssim
     \delta \|e_h\|_{H^1(\Omega)}^2 
     + \delta^{-1} \Big( h^2 +  h^{(4-d)} \Big)
     \| u \|^2_{W^2_4(\Omega)} \| p \|^2_{W^2_4(\Omega)}            
    \end{align}
where in (\ref{appendix:I1-b}) we used the a priori error estimates (\ref{eq:error_estimate_u}) and (\ref{eq:error_estimate_p}), with 
$p=2$ for the first two terms and with 
$p=4$ for the third term, and the Sobolev embedding theorem to 
conclude that 
\begin{equation}
\|\nabla v \|_{L^\infty(\Omega)} \lesssim \| v \|_{W^2_4(\Omega)} 
\end{equation}    
for $v=p$ and $v = u$ since $d\leq 3$, and finally in (\ref{appendix:I1-c}) 
we used the basic bound $\|v\|_{W^2_2(\Omega)} \lesssim \|v\|_{W^2_4(\Omega)}$ 
for $v=p$ and $v=u$.

\paragraph{Term $\boldsymbol{II_2}$.} Using the 
same approach as for Term $II_2$ (with $B$ replaced by 
$\nabla \cdot e_h$) we obtain
\begin{align}
II_2&\lesssim 
 \delta \|e_h\|_{H^1(\Omega)}^2 
     + \delta^{-1} h
     \| u \|^2_{W^2_4(\Omega)} \| p \|^2_{W^2_4(\Omega)} 
\end{align}

\paragraph{Term $\boldsymbol{II_3}$.}  Using H\"older's inequality  
     \begin{align}
            II_3&=\int_\Omega(\nabla\cdot e_h)f(p-p_h)\dx
\\            
            &\leq \|\nabla\cdot e_h\|_{L^2(\Omega_0)}\|f\|_{L^4(\Omega)}
            \|p-p_h \|_{L^4(\Omega)} 
            \\
            &\lesssim \|e_h\|_{H^1(\Omega_0)} \|f\|_{L^4(\Omega)} 
            h^{(8-d)/4}\|p\|_{W^2_4(\Omega)}
            \\
            &\lesssim \delta\|e_h\|_{H^1(\Omega_0)}^2 + \delta^{-1} h^{(8-d)/2}\|f\|_{L^4(\Omega)}^2
            \|p\|_{W^2_4(\Omega)}^2 
        \end{align}

\paragraph{Term $\boldsymbol{II_{4}}$.} Using 
the conjugate rule followed by H\"older's inequality, 
    \begin{align} 
            \int_{\Gamma}(\nabla_\Gamma\cdot e_h)(u^2-u_h^2))\ds 
            &=\int_{\Gamma}(\nabla_\Gamma\cdot e_h)(u+u_h)(u-u_h))\ds \\
            &\leq\|\nabla_\Gamma\cdot e_h\|_{L^2(\Gamma)}\|u+u_h\|_{L^4(\Gamma)}\|u-u_h\|_{L^4(\Gamma)} \\
            &\leq \delta h\|\nabla_\Gamma\cdot e_h\|_{L^2(\Gamma)}^2 
            +\delta^{-1} h^{-1} \|u+u_h\|_{L^4(\Gamma)}^2\|u-u_h\|_{L^4(\Gamma)}^2 
            \\
            &\leq \delta \|e_h\|_{H^1(\Omega_0)}^2 
            + \delta^{-1} h^{(5-d)/2}\|u\|_{W^2_4(\Omega)}^4
    \end{align}
Here we used the trace inequality 
\begin{equation}
\|v\|_{L^p(\Gamma)}
\leq 
\|v\|_{L^p(\partial \Omega)}
\lesssim
\|v\|_{L^p(\Omega)}^{1-1/p} \| v \|_{W^1_p(\Omega)}^{1/p} \qquad v \in W^1_p(\Omega)
\end{equation}
with $p=4$ and $v= u - u_h$ followed by the a priori error estimate (\ref{eq:error_estimate_u}) to conclude that
\begin{align}\label{eq:I4-a}
\|u-u_h\|_{L^4(\Gamma)}
&\lesssim 
\|u-u_h\|_{L^4(\Omega)}^{3/4} \| u - u_h \|_{W^1_4(\Omega)}^{1/4}
\\
&\lesssim
(h^{(8-d)/4})^{3/4} (h^{(4-d)/4})^{1/4} \| u \|_{W^2_4(\Omega)}
\\ \label{eq:I4-d}
&\lesssim 
h^{(7-d)/4} \| u \|_{W^2_4(\Omega)}
\end{align}
and the following estimate
\begin{align}
\|u+u_h\|_{L^4(\Gamma)} 
&\lesssim 
\|u\|_{L^4(\Gamma)} + 
\|u-u_h\|_{L^4(\Gamma)}
\\ 
&\lesssim 
\|u\|_{L^4(\Gamma)} + 
\|u-u_h\|_{W^1_4(\Omega)}
\\
&\lesssim 
\|u\|_{L^4(\Gamma)} + 
h^{(4-d)/4}\|u\|_{W^2_4(\Omega)}
\\
&
\lesssim \|u\|_{W^2_4(\Omega)}
\end{align}
which holds since $h\in (0,h_0]$.

\paragraph{Term $\boldsymbol{II_{5}}$.} Using H\"older's inequality 
    \begin{align}
            II_{5}&=\int_{\Gamma}(\nabla_\Gamma\cdot e_h)g_N(p-p_h)\ds \\
            &\leq\|\nabla_\Gamma\cdot e_h\|_{L^2(\Gamma)}\|g_N\|_{L^4(\Gamma)}\|p-p_h\|_{L^4(\Gamma)} \\
            &\leq \delta h\|\nabla_\Gamma\cdot e_h\|_{L^2(\Gamma)}^2 
            + \delta h^{-1}\|g_N\|_{L^4(\Gamma)}^2\|p-p_h\|_{L^4(\Gamma)}^2 
            \\ \label{eq:I5-a}
            &\leq \delta \|e_h\|_{H^1(\Omega_0)}^2 
               + h^{(5-d)/2}\|g_N\|_{L^4(\Gamma)}^2\|p\|_{W^2_4(\Omega)}^2            
    \end{align}
where in (\ref{eq:I5-a}) the first term was estimated using 
the inverse estimate 
\begin{equation}
h \| v \|^2_{L^2(\Gamma \cap T)} \lesssim \| v \|^2_{H^1(T)}\qquad v \in P_1(T)
\end{equation}
for $T\in \mcTh$ such that $\Gamma \cap T \neq \emptyset$ with $v=e_h$, 
and for the second term we used the estimate 
\begin{align}
\|p-p_h\|_{L^4(\Gamma)}
&\lesssim 
h^{(7-d)/4} \|p\|_{W^2_4(\Omega)}
\end{align}
which follows in the same way as in (\ref{eq:I4-a}-\ref{eq:I4-d}).
    
\bibliographystyle{abbrv}
\bibliography{referencesCutFEM}

\end{document}